\documentclass{amsart}

\usepackage{amssymb}
\usepackage{amsmath}
\usepackage{amsfonts}
\usepackage{geometry}
\usepackage{bbm}
\usepackage{hyperref}
\usepackage{mdef}
\usepackage{tikz}
\usetikzlibrary{matrix,arrows,decorations.pathmorphing}

\newcounter{cprop}[section]

\newtheorem{theorem}[cprop]{Theorem}

\theoremstyle{plain}

\newtheorem{corollary}[cprop]{Corollary}

\newtheorem{lemma}[cprop]{Lemma}
\newtheorem{proposition}[cprop]{Proposition}

\numberwithin{equation}{section}

\theoremstyle{definition}

\newtheorem{example}[cprop]{Example}

\theoremstyle{remark}
\newtheorem{remark}[cprop]{Remark}

\email{riedel@math.tu-berlin.de}
 \thanks{{\bf Acknowledgements:} SR is grateful for support from the European Research Council under the European Union’s Seventh Framework Programme ERC grant agreement nr. 258237. Financial support by the DFG via Research Unit FOR 2402 is gratefully acknowledged. Further, the author would like to thank Peter Friz and Michael Scheutzow for valuable discussions and comments. }


\begin{document}
\title[Transportation--cost inequalities, Gaussian processes]{Transportation--cost inequalities for diffusions driven by Gaussian processes}
\author{Sebastian Riedel}
\address{Sebastian Riedel \\
Institut f\"ur Mathematik, Technische Universit\"at Berlin, Germany}

\begin{abstract}
  We prove transportation--cost inequalities for the law of SDE solutions driven by general Gaussian processes. Examples include the fractional Brownian motion, but also more general processes like bifractional Brownian motion. In case of multiplicative noise, our main tool is Lyons' rough paths theory. We also give a new proof of Talagrand's transportation--cost inequality on Gaussian Fr\'echet spaces. We finally show that establishing transportation--cost inequalities implies that there is an easy criterion for proving Gaussian tail estimates for functions defined on that space. This result can be seen as a further generalization of the ``generalized Fernique theorem'' on Gaussian spaces \cite[Theorem 11.7]{FH14} used in rough paths theory.

\end{abstract}

\date{\today}
\maketitle

\section*{Introduction}

Transportation--cost inequalities can be seen as a functional approach to the concentration of measure phenomenon (cf. Ledoux's work \cite{Led01} for an introduction to the theory of measure concentration and the work \cite{GL10} by Gozlan and L\'eonard for an overview to transport inequalities). They are usually of the following form: Let $(E,d)$ be a metric space and let $P(E)$ denote the set of probability measures on the Borel sets of $E$. We say that a  \emph{$p$-transportation--cost inequality} holds for a measure $\mu \in P(E)$ if there is a constant $C$ such that
\begin{align}\label{eqn:transp_cost_metric_space}
 \mathcal{W}_p(\nu,\mu) \leq \sqrt{C H(\nu\,|\,\mu)}
\end{align}
holds for all $\nu \in P(E)$. Here $\mathcal{W}_p(\nu,\mu)$ denotes the Wasserstein $p$-distance 
\begin{align*}
 \mathcal{W}_p(\nu,\mu) = \inf_{\pi \in \Pi(\nu,\mu)}\left( \int_{E \times E} d(x,y)^p\, d\pi(x,y) \right)^{\frac{1}{p}}
\end{align*}
where $\Pi(\nu,\mu)$ is the set of all probability measures on the product space $E \times E$ with marginals $\nu$ resp. $\mu$, and $H(\nu\,|\,\mu)$ is the relative entropy (or Kullback--Leibler divergence) of $\nu$ with respect to $\mu$, i.e.
\begin{align*}
 H(\nu\,|\,\mu) = 
 \begin{cases}
  \int \log\left(\frac{d\nu}{d\mu}\right)\,d\nu &\text{if }\nu \ll \mu \\
  + \infty &\text{otherwise.}
 \end{cases}
\end{align*}
If \eqref{eqn:transp_cost_metric_space} holds, we will say that $T_p(C)$ holds for the measure $\mu$. 

Inequalities of type \eqref{eqn:transp_cost_metric_space} were first considered by Marton (cf. \cite{Mar86}, \cite{Mar96}).
The cases ``$p = 1$'' and ``$p = 2$'' are of special interest: The $1$-transportation--cost inequality, i.e. the weakest form of \eqref{eqn:transp_cost_metric_space}, is actually equivalent to Gaussian concentration as it was shown by Djellout, Guillin and Wu in \cite{DGW04} (using preliminary results by Bobkov and G\"otze obtained in \cite{BG99}). The $2$-transportation--cost inequality was first proven by Talagrand for the Gaussian measure on $\R^d$ in \cite{Tal96} with the sharp constant $C = 2$ (for this reason it is also called \emph{Talagrand's transportation--cost inequality}). $T_2(C)$ is particularly interesting since it has the \emph{dimension--free tensorization property}: If $T_2(C)$ holds for two measures $\mu_1$ and $\mu_2$, it also holds for the product measure $\mu_1 \otimes \mu_2$ \emph{for the same constant $C$} (see also \cite{GL07} for a general account on tensorization properties for transportation--cost inequalities), and this property yields the dimension--free concentration of measure property for $\mu$. Gozlan realized in \cite{Goz09} that also the converse is true: If $\mu$ possesses the dimension--free concentration of measure property, $T_2(C)$ holds for $\mu$. We also remark that the $2$--transportation-cost inequality gained much attention because it is intimately linked to other famous concentration inequalities, notably to the logarithmic Sobolev inequality: In their celebrated paper \cite{OV00}, Otto and Villani showed that in a smooth Riemannian setting, the logarithmic Sobolev inequality implies the $2$-transportation--cost inequality. Since then, this result has been generalized in several directions, see e.g. the recent work of Gigli and Ledoux \cite{GL13} and the references therein.

In this work, we will mainly study transportation--cost inequalities for the law of a continuous diffusion $Y$ induced by a stochastic differential equation (SDE) driven by a general Gaussian process, i.e. $Y \colon [0,T] \to \R^d$ solves
\begin{align}\label{eqn:SDE_intro}
 dY_t = b(Y_t)\, dt + \sum_{i = 1}^m \sigma_i(Y_t)\circ dX_t^i; \qquad Y_0 = \xi \in \R^d, \quad t\in [0,T]
\end{align}
where $X = (X^1,\ldots,X^m) \colon [0,T] \to \R^m$ is a continuous Gaussian process and $b, \sigma_1,\ldots,\sigma_m$ are vectorfields in $\R^d$. Of course, the equation \eqref{eqn:SDE_intro} needs an interpretation in a non-martingale setting in which It\=o's theory is not applicable. However, in the case when $X$ is a Brownian motion, the equation \eqref{eqn:SDE_intro} can be solved using It\=o's framework, and transportation--cost inequalities were studied in many works: In this context, $T_1(C)$ was first established for the law of $Y$ with respect to the uniform metric by Djellout, Guillin and Wu in \cite{DGW04}. In the same work, also $T_2(C)$ was proven, but for the weaker $L^2$-metric only. Under stronger assumptions on the equation (which guarantee in particular the existence of a unique invariant probability measure and exponential convergence towards it), Wu and Zhang proved in \cite{WZ04} that also $T_2(C)$ holds for the uniform metric. \"Ust\"unel finally proved $T_2(C)$ for the uniform metric 
in \cite{Ust12} in the most general form. However, replacing the Brownian motion by another Gaussian process, not much is known. To the authors knowledge, the only process which was studied, up to a certain extend, is the \emph{fractional Brownian motion} (fBm). By definition, a fBm with Hurst parameter $H \in (0,1)$ is a centered Gaussian process with covariance
\begin{align*}
 R(s,t) = \frac{1}{2} \left( |t|^{2H} + |s|^{2H} - |t - s|^{2H} \right),
\end{align*}
and it is easily seen that we obtain the usual Brownian motion for $H = 1/2$. However, for $H \neq 1/2$ this process is neither a semimartingale nor a Markov process. Guendouzi shows $T_1(C)$ for the $L^1$-metric for a mixed SDE involving a fBm with Hurst parameter $H > 1/2$ in \cite{Gue12}. Saussereau studies more general equations in \cite{Sau12} and shows $T_1(C)$ and $T_2(C)$ also for the uniform metric in particular situations. However, all equations he considers are either driven by a fBm with Hurst parameter $H > 1/2$, have additive noise or are one-dimensional. In fact, all these examples have something in common. Namely, it is known that in these cases, the solution to \eqref{eqn:SDE_intro} is a continuous function of the driving process \emph{path-by-path}. This is not true in the general case of \eqref{eqn:SDE_intro} (and already fails, for instance, for the usual Brownian motion). For studying the equation \eqref{eqn:SDE_intro} in full generality, one needs further ingredients, and we will use Lyon's \emph{rough paths theory} to achieve this goal. Let us mention that our results imply those obtained in  \cite{Sau12} in case of fBm.

There is a further challenge when studying transportation--cost inequalities for solutions to \eqref{eqn:SDE_intro} for general Gaussian processes $X$. The standard tool to establish transportation-cost inequalities, following \cite{FU04} and \cite{DGW04}, is to use the \emph{Girsanov transformation}. In a non-martingale framework, this argument completely breaks down. In case of the fBm, it can still be applied up to a certain point due to the Mandelbrot--van Ness representation of the fBm as a stochastic integral with respect to standard Brownian motion \cite{MVN68}. However, there are many Gaussian processes (and we will encounter a class of them in the forthcoming Example \ref{ex:bifBm}) where such a representation is simply not known. Our approach can be seen as an attempt to prove concentration inequalities for diffusions \emph{avoiding} the Girsanov transformation.

Let us explain our strategy and the contribution of this work. In Section \ref{sec:transp_cost_gaussian_space}, we consider transportation--cost inequalities on infinite dimensional Gaussian spaces. In turns out that in this framework, the quadratic transport inequality even holds for the \emph{Cameron--Martin metric}, which is defined as follows: If $\mathcal{H}$ denotes the Cameron--Martin space associated to a Gaussian measure $\gamma$, set
\begin{align}\label{eqn:def_CM_metric}
	d_{\mathcal{H}}(x,y) = \begin{cases}
		|x - y|_{\mathcal{H}} &\text{if } x - y \in \mathcal{H} \\
		+ \infty &\text{otherwise.}
	\end{cases}
\end{align}
The fact that a transport inequality holds for $\gamma$ and this metric should be surprising at first sight since it is known that for infinite dimensional spaces, the Hilbert space $\mathcal{H}$ has $\gamma$-measure $0$; in other words, $d_{\mathcal{H}}(x,y) = \infty$ ``very often''. In this form, the quadratic transport inequality was first proven by Feyel and \"Ust\"unel on Gaussian Banach spaces in \cite[Theorem 3.1]{FU04} using the Girsanov transformation (cf. also Gentil's PhD thesis \cite{GenPhd01}). The proof we give does not rely on the Girsanov transformation and holds even in Fr\'echet spaces (cf. Theorem \ref{thm:talagrand_strong_form_linear}). Our main tool for proving transport inequalities for solutions to \eqref{eqn:SDE_intro} will be a \emph{contraction principle}, first proven\footnote{In the context of measure concentration, this contraction principle already appeared earlier in a work by Maurey for of infimal concolution inequalities, see \cite[Lemma 2]{Mau91}.} by Djellout, Guillin and Wu in \cite[Lemma 2.1]{DGW04} (we state a slightly more general version in the appendix, cf. Lemma \ref{lemma:transp_ineq_trans_under_loc_lip_maps}), which states that transport inequalities are stable under Lipschitz maps. Together with our result about transport inequalities on Gaussian spaces, all we need to establish is Lipschitzness of the solution map $X(\omega) \mapsto Y(\omega)$ for equation \eqref{eqn:SDE_intro}. This is usually true for additive noise, and we study this case in Section \ref{sec:sde_add_noise} first. Interestingly, due to the strong form of the Gaussian transportation--cost inequality, we obtain such inequalities for the law of $Y$ for metrics which are much larger than the uniform metric (cf. Theorem \ref{thm:transp_cost_ineq_add_noise_case} and the discussion in Example \ref{ex:bifBm}) in the case of $b$ in \eqref{eqn:SDE_intro} being Lipschitz continuous. We further study the case where $b$ only satisfies a one-sided Lipschitz condition in Theorem \ref{thm:transp_cost_ineq_add_noise_one_sided_lip}. We proceed with the multiplicative noise case in Section \ref{sec:sde_mult_noise}. As already mentioned, here we cannot expect the solution map to be Lipschitz continuous anymore in the usual topologies. The key idea is to use the \emph{rough path factorization}: Instead of studying the map $X(\omega) \mapsto Y(\omega)$ directly, we consider an intermediate step; namely, we decompose this map as
\begin{align}\label{eqn:rp_factorization}
 X(\omega) \overset{S}{\mapsto} \mathbf{X}(\omega) \overset{\mathbf{I}}{\mapsto} Y(\omega).
\end{align}
The map $S$ is called \emph{lift map}, and it takes a Gaussian trajectory and maps it to a rough path. It is not continuous, but easy to analyze. The map $\mathbf{I}$ is called \emph{It\=o-Lyons map}, and it is known to be continuous, and even locally Lipschitz continuous in rough paths topology (in fact, this result can be seen as the main theorem in rough paths theory). The point now is that $S$ can be shown to be locally Lipschitz continuous \emph{from $\mathcal{H}$ to a rough paths space}, hence the decomposition, seen as a map from $\mathcal{H}$ to the space of continuous paths, \emph{is} locally Lipschitz continuous. The contraction principle allows us to conclude $T_{2-\varepsilon}(C)$ for any $\varepsilon > 0$, cf. Theorem \ref{thm:main_result_multiplicative_sdes} (the $\varepsilon$-correction stems from the fact the we only have \emph{local} Lipschitzness). Finally, we discuss the link between $T_p(C)$ and tail estimates for functions in Section \ref{sec:tail_estimates_for_functionals} and establish 
a link between $T_p(C)$ and the \emph{generalized Fernique theorem} (cf. \cite[Theorem 17]{DOR15}, \cite[Theorem 11.7]{FH14} and \cite{FO10}) which is of fundamental importance in rough paths theory (cf. \cite[Chapter 11]{FH14}). This section does not depend on the former ones and may be of independent interest.

Let us finally mention that we think that our approach can be carried over to SDEs in infinite dimensions, i.e. to \emph{stochastic partial differential equations}, in particular to those considered in Hairer's theory of \emph{regularity structures} \cite{Hai14} or Gubinelli-Imkeller-Perkowski's approach using \emph{paracontrolled distributions} \cite{GIP15}. Indeed, in both theories, it was understood that (after a possible renormalization), singular equations like the KPZ-equation (cf. also \cite{Hai13}) often have a similar factorization as in \eqref{eqn:rp_factorization}, and this was the basic ingredient we needed for ordinary SDEs as well.

\subsection*{Notation}

If $(X,\mathcal{F})$ is a measurable space, $P(X)$ denotes the set of all probability measures defined on $\mathcal{F}$. If $X$ is a topological space, $\mathcal{F}$ will be usually be the Borel $\sigma$-algebra $\mathcal{B}(X)$. If $X$ and $Y$ are measurable spaces and $\nu \in P(X)$, $\mu \in P(Y)$, then $\Pi(\nu,\mu)$ denotes the set of all product measures on $X \times Y$ with marginals $\nu$ resp. $\mu$. 
If $[S,T]$ is any interval in $\R$, we write $\mathcal{P}([S,T])$ for the set of all finite partitions of $[S,T]$ of the form $S = t_0 < t_1 < \ldots < t_M = T$, $M \in \N$. If $x,y \colon [S,T] \to (B,\| \cdot \|)$ are paths with values in a normed space and $p\geq 1$, we define $p$-variation seminorm and pseudometric as
\begin{align}\label{eqn:def_pvar_dist}
 \|x \|_{p-\text{var};[S,T]} := \sup_{D \in \mathcal{P}([S,T])} \left( \sum_{t_i \in D} \| x_{t_{i+1}} - x_{t_i} \|^p \right)^{\frac{1}{p}}; \qquad d_{p-\text{var};[S,T]}(x,y) := \|x - y \|_{p-\text{var};[S,T]}.
\end{align}
If the time horizon is clear from the context, we sometimes omit the subindex $[S,T]$ in the notation. The set of all continuous paths $x \colon [S,T] \to B$ with $\| x\|_{p-\text{var};[S,T]} < \infty$ is denoted by $C^{p-\text{var}}([S,T];B)$ and we also define $C^{p-\text{var}}_{\xi}([S,T];B) := \{ x \in C^{p-\text{var}}([S,T];B)\,:\, x_S = \xi \}$ for some $\xi \in B$. If $B$ is a Banach space, $C^{p-\text{var}}_{0}([S,T];B)$ is also a Banach space with the norm $\| \cdot \|_{p-\text{var}}$.

\section{Transportation inequality on a Gaussian space}\label{sec:transp_cost_gaussian_space}

In this section, we give a proof of $T_2(2)$ on Gaussian spaces for the Cameron--Martin metric defined in \eqref{eqn:def_CM_metric}, a result which was first proven on Banach spaces by Feyel and \"Ust\"unel \cite[Theorem 3.1]{FU04} using the Girsanov transformation. Our strategy will be to ``approximate'' the infinite dimensional space by finite dimensional ones on which we know from Talagrand's original result that $T_2(2)$ holds.

We start with an abstract approximation result.

\begin{lemma}\label{lemma:approx_cost_functions}
 Let $X$ and $Y$ be Polish spaces and let $(\mu_n)$ and $(\nu_n)$ be sequences of probability measures on $X$ resp. $Y$ which converge weakly to some probability measures $\mu$ resp. $\nu$. Let $c_n \colon X \times Y \to [0,\infty)$ be a nondreasing sequence of bounded, continuous functions such that $c_n \nearrow c$ pointwise where $c \colon X \times Y \to [0,\infty]$. Then, along a subsequence,
 \begin{align*}
  \liminf_{k \to \infty} \inf_{\pi \in \Pi(\mu_{n_k},\nu_{n_k})} \int_{X \times Y} c_{n_k}(x,y)\,d\pi(x,y) \geq \inf_{\pi \in \Pi(\mu,\nu)} \int_{X \times Y} c(x,y)\,d\pi(x,y).
 \end{align*}
\end{lemma}

\begin{proof}
 For $\pi \in P(X \times Y)$, set
 \begin{align*}
  I_n(\pi) := \int_{X \times Y} c_n(x,y)\,d\pi(x,y) \quad \text{and} \quad I(\pi) = \int_{X \times Y} c(x,y)\,d\pi(x,y).
 \end{align*}
 From continuity of the $c_n$, we know from \cite[Theorem 1.3]{Vil03} that there are measures $\pi_n \in \Pi(\mu_n,\nu_n)$ such that
 \begin{align*}
  I_n(\pi_n) = \inf_{\pi \in \Pi(\mu_n,\nu_n)} I_n(\pi)
 \end{align*}
 for all $n \geq 1$. We claim that the sequence $(\pi_n)$ is tight in $P(X \times Y)$. Indeed: Let $\varepsilon > 0$ be given. From Prokhorov's theorem, the sequences $(\mu_n)$ and $(\nu_n)$ are tight, therefore we can find compact sets $K_1 \subset X$ and $K_2 \subset Y$ such that
 \begin{align*}
  \mu_n(K_1) \geq 1 + \varepsilon/2 \quad \text{and} \quad \nu_n(K_2) \geq 1 + \varepsilon/2
 \end{align*}
 for all $n \geq 1$. This implies that
 \begin{align*}
  \pi_n(K_1 \times K_2) \geq 1 + \varepsilon
 \end{align*}
 for all $n \geq 1$ which shows tightness. Using agains Prokhorov's theorem, we know that there is a subsequence for which $\pi_{n_k} \to \pi^*$ weakly in $P(X \times Y)$ for $k \to \infty$. Let $f \colon X \to \R$ be a bounded, continuous function. From weak convergence,
 \begin{align*}
  \lim_{k \to \infty} \int_X f(x)\, d\mu_{n_k}(x) =  \int_X f(x)\, d\mu(x).
 \end{align*}
 Furthermore,
 \begin{align*}
  \lim_{k \to \infty} \int_{X \times Y} f(x)\, d\pi_{n_k}(x,y) =  \int_{X \times Y} f(x)\, d\pi^*(x,y) = \int_X f(x)\, d\pi^*(x,Y)
 \end{align*}
 which implies that $\pi^*(\cdot,Y) = \mu$. Similarly, $\pi^*(X,\cdot) = \nu$, and we have thus shown that $\pi^* \in \Pi(\mu,\nu)$. From monotonicity, whenever $n\geq m$, $I_n(\pi_n) \geq I_m(\pi_n)$, and therefore
 \begin{align*}
  \liminf_{k \to \infty} I_{n_k}(\pi_{n_k}) \geq \liminf_{k \to \infty} I_m(\pi_{n_k}) \geq I_m(\pi_*)
 \end{align*}
 for every $m \geq 1$. Monotone convergence gives
 \begin{align*}
  \lim_{m \to \infty} I_m(\pi_*) = I(\pi_*),
 \end{align*}
 and thus
 \begin{align*}
  \liminf_{k \to \infty} I_{n_k}(\pi_{n_k}) \geq \lim_{m \to \infty} I_m(\pi_*) = I(\pi_*) \geq \inf_{\pi \in \Pi(\mu,\nu)} I(\pi).
 \end{align*}
 
\end{proof}

In the following, we aim to consider Gaussian measures on linear spaces. Typically, one assumes that the space should be \emph{locally convex}, i.e. its topology is generated by family of seminorms separating points (cf. \cite[Chapter 2 and Appendix A]{Bog98}). It will be convenient for us to assume that the space is also \emph{Polish}, i.e. separable and completely metrizable. Such spaces are also called separable \emph{Fr\'echet spaces}. A \emph{Gaussian Fr\'echet spaces} is a triplet $(F, \mathcal{H},\gamma)$ where $F$ is a separable Fr\'echet spaces, $\gamma$ is a Gaussian measure on the Borel $\sigma$-field $\mathcal{B}(F)$ and $\mathcal{H}$ denotes the \emph{Cameron-Martin space} which is a separable Hilbert space $(\mathcal{H},\langle \cdot,\cdot \rangle)$ lying in $F$ (cf. \cite[Section 2.4]{Bog98} for the precise definition of the Cameron Martin space and \cite[Chapter 2 and 3]{Bog98} for further properties). The induced norm on $\mathcal{H}$ will be denoted by $|\cdot |_{\mathcal{H}}$. Recall the 
definition of the Cameron--Martin metric given in \eqref{eqn:def_CM_metric}.

The following theorem is the main result from this section.

\begin{theorem}\label{thm:talagrand_strong_form_linear}
  Let $(F, \mathcal{H},\gamma)$ be a Gaussian Fr\'echet space. Then for any $\nu \in P(F)$,
 \begin{align*}
   \inf_{\pi \in \Pi(\nu,\gamma)} \int_{F \times F} d_{\mathcal{H}}(x,y)^2\, d\pi(x,y) \leq 2 \, H(\nu\,|\,\gamma).
 \end{align*}
\end{theorem}

\begin{proof}
 Note that for every $h, k \in \mathcal{H}$, there are elements $\hat{h}, \hat{k} \in F^*$ such that
 \begin{align*}
  \langle h, k \rangle_{\mathcal{H}} = \langle \hat{h}, \hat{k} \rangle_{L^2(\gamma)} = \int_F \hat{h}(x) \hat{k}(x) \, d \gamma(x),
 \end{align*}
 cf. \cite[Section 2.4 and 3.2.3 Theorem]{Bog98}. Let $(e_n)$ be an orthonormal basis of $\mathcal{H}$. Define
 \begin{align*}
  H_n := \operatorname{span}\{e_1,\ldots, e_n\}
 \end{align*}
 and $p_n \colon F \to H_n$ by
 \begin{align*}
  p_n(x) = \sum_{k = 1}^n \hat{e}_k(x) e_k.
 \end{align*}
 Note that for $h \in \mathcal{H}$, $\hat{e}_k(h) = \langle e_k, h \rangle_{\mathcal{H}}$ by \cite[2.10.5 Lemma]{Bog98}, thus
 \begin{align*}
  p_n(h) = \sum_{k = 1}^n \langle e_k, h \rangle_{\mathcal{H}} e_k.
 \end{align*} 
 We equip the space $H_n$ with the scalar product
 \begin{align*}
  \langle v,w \rangle_{H_n} = \sum_{k = 1}^n \lambda_k \mu_k \quad \text{if} \quad v = \sum_{k = 1}^n \lambda_k e_k,\ w = \sum_{k = 1}^n \mu_k e_k.
 \end{align*}
 Note that with this definition, $\langle p_n(h), p_n(k) \rangle_{H_n} = \langle p_n(h), p_n(k) \rangle_{\mathcal{H}}$. Consider the image measure $\tilde{\gamma}_n := \gamma \circ p_n^{-1}$. Then $(H_n,\tilde{\gamma}_n)$ is a finite dimensional Gaussian space, and we know from Talagrand's result that $T_2(2)$ holds here. Consider the inclusion maps $\iota_n \colon H_n \hookrightarrow F$ and set $\gamma_n := \tilde{\gamma}_n \circ \iota_n^{-1}$. By the contraction principle in Lemma \ref{lemma:transp_ineq_trans_under_loc_lip_maps}, we see that for every $\nu \in P(F)$
\begin{align*}
   \inf_{\pi \in \Pi(\nu,\gamma_n)} \int_{F \times F} \tilde{d}_{n}(x,y)^2\, d\pi(x,y) \leq 2 \, H(\nu\,|\,\gamma_n)
 \end{align*}
 holds for all $n \geq 1$ where
 \begin{align*}
  \tilde{d}_{n}(x,y) = |p_n(x) - p_n(y)|_{\mathcal{H}}.
 \end{align*}
 Set $d_{n}(x,y) := \tilde{d}_{n}(x,y) \wedge n$. Since $d_n \leq \tilde{d}_n$, also 
 \begin{align}\label{eqn:tci_interm}
   \inf_{\pi \in \Pi(\nu,\gamma_n)} \int_{F \times F} d_{n}(x,y)^2\, d\pi(x,y) \leq 2 \, H(\nu\,|\,\gamma_n)
 \end{align}
 holds for every $\nu \in P(F)$ and $n \geq 1$. We collect some facts about the functions $d_n$. First, it is clear by definition that all $d_n \colon F \times F \to [0,\infty)$ are bounded and continuous. Furthermore, for fixed $x,y \in F$,
 \begin{align*}
 \tilde{d}_n(x,y)^2 =  \sum_{k = 1}^n | \hat{e}_k(x - y) |^2 \leq \sum_{k = 1}^{n+1} | \hat{e}_k(x - y) |^2 = \tilde{d}_{n+1}(x,y)^2
  \end{align*}
  which shows that the sequence $(d_n)$ is nondecreasing. We claim that $d_n \nearrow d_{\mathcal{H}}$ pointwise for $n \to \infty$. Indeed, if $x-y \in \mathcal{H}$,
  \begin{align*}
   \lim_{n\to\infty} d_n(x,y)^2 = \sum_{k = 1}^\infty |\langle e_k, x - y \rangle_{\mathcal{H}}|^2  = \left|x - y \right|^2_{\mathcal{H}} = d_{\mathcal{H}}(x,y)^2
  \end{align*}
  by Parseval's identity. Conversely, if $\lim_{n \to \infty} d_n(x,y) < \infty$ for some $x,y \in F$, we may define
  \begin{align*}
   \sum_{k = 1}^\infty  \hat{e}_k(x - y) e_k =: z \in \mathcal{H}.
  \end{align*}
  This implies that
  \begin{align*}
   \sum_{k = 1}^\infty \hat{e}_k(x - y) e_k = \sum_{k = 1}^\infty \hat{e}_k(z)  e_k
  \end{align*}
  and applying $\hat{e}_k$ on both sides shows that $\hat{e}_k (x - y) = \hat{e}_k(z)$ holds for every $k\in \N$. Hence $x - y = z \in \mathcal{H}$ and we have shown the claim. Next, we show that $\gamma_n \to \gamma$ weakly for $n \to \infty$. Let $g \colon F \to \R$ be a bounded, continuous function. Then
  \begin{align*}
   \int_F g(x)\, d\gamma_n(x) = \int_F g\left( \sum_{k = 1}^n \hat{e}_k(x) e_k \right)\, d\gamma(x).
  \end{align*}
  We know from \cite[3.5.1 Theorem]{Bog98} that $\sum_{k = 1}^{\infty} \hat{e}_k(x) e_k = x$ $\gamma$-almost surely, hence by dominated convergence,
  \begin{align*}
   \int_F g(x)\, d\gamma_n(x) \to \int_F g(x)\, d\gamma(x)
  \end{align*}
  for $n \to \infty$ which shows weak convergence. Choose any $\nu \in P(F)$ with $ \nu \ll \gamma$. Set $f := \frac{d \nu}{d \gamma}$ and define $d\nu_n := f\, d\gamma_n$. From \eqref{eqn:tci_interm}, we have
  \begin{align*}
   \inf_{\pi \in \Pi(\nu_n,\gamma_n)} \int_{F \times F} d_{n}(x,y)^2\, d\pi(x,y) \leq 2 \int_F f \log f\, d \gamma_n
  \end{align*}
  for every $n \in \N$. Assume first that $f$ is bounded and continuous. In this case, we have $\nu_n \to \nu$ weakly for $n \to \infty$ and we can use Lemma \ref{lemma:approx_cost_functions} for the left hand side and weak convergence for the right hand side of the above inequality to conclude that indeed
 \begin{align*}
   \inf_{\pi \in \Pi(\nu,\gamma)} \int_{F \times F} d_{\mathcal{H}}(x,y)^2\, d\pi(x,y) \leq 2 \int_F f \log f\, d \gamma
 \end{align*}  
  holds for every $\nu \in P(F)$ with bounded, continuous density. Next, we extend this result to arbitrary density functions. By a result of Wi{\'s}niewski \cite[Theorem 1]{Wis94}, for every measurable map $f \colon F \to \R$ there exists a sequence of continuous functions $(f_n)$ such that $f_n \to f$ $\gamma$-almost surely. Assume first that the density $f$ is bounded by some $C > 0$. Let $(f_n)$ be a sequence of continuous functions converging $\gamma$-a.s. to $f$. We may assume w.l.o.g. that $0 \leq f_n \leq C$ for all $f_n$, otherwise we replace each $f_n$ by $(f_n \wedge C) \vee 0$. Set $ \alpha_n := \| f_n \|_{L^1(\gamma)}$ and $d\nu_n := (f_n/\alpha_n)\, d \gamma$. We have shown that for every $n \in \N$,
  \begin{align*}
   \inf_{\pi \in \Pi(\nu_n,\gamma)} \int_{F \times F} d_{\mathcal{H}}(x,y)^2\, d\pi(x,y) \leq 2 \int_F (f_n/\alpha_n) \log (f_n/\alpha_n) \, d \gamma.
  \end{align*}
  The above inequality implies that also
   \begin{align*}
   \inf_{\pi \in \Pi(\nu_n,\gamma)} \int_{F \times F} d_m(x,y)^2\, d\pi(x,y) \leq 2\int_F (f_n/\alpha_n) \log (f_n/\alpha_n) \, d \gamma
  \end{align*}
  holds for every fixed $n, m \in \N$. From Lebesgue's dominated convergence theorem, we can conclude that $\nu_n \to \nu$ weakly and 
  \begin{align*}
   \int_F (f_n / \alpha_n) \log (f_n / \alpha_n)\, d \gamma \to \int_F f \log f\, d \gamma
  \end{align*}
 for $n \to \infty$. We can use Lemma \ref{lemma:approx_cost_functions} again (now for the fixed cost function $d_m$) to see that
  \begin{align}\label{eqn:transp_cost_d_m}
   \inf_{\pi \in \Pi(\nu,\gamma)} \int_{F \times F} d_m(x,y)^2\, d\pi(x,y) \leq 2 \int_F f \log f\, d \gamma
  \end{align}
  holds for every $m \in \N$ and every bounded density $f$. Now let $f$ be an arbitrary density function. Set $f_n := f \wedge n$, $ \alpha_n := \| f_n \|_{L^1(\gamma)}$ and $d\nu_n := (f_n/\alpha_n)\, d \gamma$. Using monotone convergence, we see that $\nu_n \to \nu$ weakly and $H(\nu_n \,|\,\gamma) \to H(\nu \,|\,\gamma)$ for $n \to \infty$. As before, Lemma \ref{lemma:approx_cost_functions} shows that \eqref{eqn:transp_cost_d_m} holds for every $ \nu \ll \gamma$ with density function $f$ and every $m \in \N$. Taking the limes inferior along a subsequence of $m$ in \eqref{eqn:transp_cost_d_m}, we can use Lemma \ref{lemma:approx_cost_functions} a fourth time to conclude the assertion of our theorem.

\end{proof}

\subsection{Banach spaces}

Let $(B,\|\cdot\|)$ be a separable Banach space and set $d_B(x,y) := \|x - y\|$. As an immediate corollary of Theorem \ref{thm:talagrand_strong_form_linear} we obtain:
\begin{corollary}\label{corollary:weak_conclusion}
 Let $(B, \mathcal{H},\gamma)$ be a Gaussian Banach space. Then for any $\nu \in P(B)$,
 \begin{align*}
  \inf_{\pi \in \Pi(\nu,\gamma)} \int_{B \times B} {d_B}(x,y)^2\,d\pi(x,y) \leq 2 \sigma^2\, H(\nu\,|\,\gamma)
 \end{align*}
 where
 \begin{align}\label{eqn:var_gauss_measure}
  \sigma^2 = \sup_{l \in B^*, \|l\| \leq 1} \int l(x)^2\,d\gamma(x) < \infty.
 \end{align}

\end{corollary}

\begin{proof}
 It is well known that $\sigma < \infty$ and that for every $h\in \mathcal{H}$ one has $\| h \| \leq \sigma |h|_{\mathcal{H}}$, cf. \cite[Chapter 4]{Led96}, which gives the claim.
\end{proof}

\subsection{Rough paths spaces}

In the case of $B = C_0([0,T],\R^d)$, Theorem \ref{thm:talagrand_strong_form_linear} immediately generalizes to rough paths spaces. Let $\gamma$ be a Gaussian measure on $B$ with corresponding Cameron--Martin space $\mathcal{H}$. For the sake of simplicity, we will assume that $\mathcal{H}$ is continuously embedded in $C_0$, otherwise we could have used a smaller space lying in $C_0$ instead. Let $\mathcal{D}$ be a rough paths space (which could either be geometric or non-geometric, a $p$-variation or an $\alpha$-H\"older rough paths space, cf. \cite{LCL07}, \cite{FV10} or \cite{FH14} for a precise definition) and assume that there is a measurable map $S \colon C_0 \to \mathcal{D}$ such that $\pi_1 \circ S = \operatorname{Id}_{C_0}$ holds where $\pi_1 \colon \mathcal{D} \to C_0$ is the projection map. The map $S$ is called a \textit{lift map}. Set $\boldsymbol{\gamma} = \gamma\circ S^{-1}$. Abusing notation, we define $d_{\mathcal{H}} \colon \mathcal{D} \times \mathcal{D} \to \R\cup\{+\infty\}$ as

\begin{align*}
 d_{\mathcal{H}}(\mathbf{x},\mathbf{y}) = \begin{cases}
                                           |\pi_1(\mathbf{x}) - \pi_1(\mathbf{y})|_{\mathcal{H}} &\text{if } \pi_1(\mathbf{x}) - \pi_1(\mathbf{y}) \in \mathcal{H} \\
  + \infty &\text{otherwise.}
                                          \end{cases}
\end{align*}

 \begin{corollary}\label{corollary:talagrand_gaussian_rp_space}
  For any $\boldsymbol{\nu} \in P(\mathcal{D})$,
 \begin{align*}
  \inf_{\boldsymbol{\pi} \in \Pi(\boldsymbol{\nu},\boldsymbol{\gamma})} \int_{\mathcal{D} \times \mathcal{D}} d_{\mathcal{H}}(\mathbf{x},\mathbf{y})^2 \, d\boldsymbol{\pi}(\mathbf{x},\mathbf{y}) \leq 2 \, H(\boldsymbol{\nu}\,|\,\boldsymbol{\gamma}).
 \end{align*}
\end{corollary}

\begin{proof}
 By definition, $d_{\mathcal{H}}(S(x),S(y)) = d_{\mathcal{H}}(x,y)$, hence $S$ is (in particular) 1-Lipschitz and the result follows from Theorem \ref{thm:talagrand_strong_form_linear} and Lemma \ref{lemma:transp_ineq_trans_under_loc_lip_maps}.
\end{proof}

\section{Applications to diffusions}\label{sec:appl_diff}

\subsection{SDEs with additive noise}\label{sec:sde_add_noise}

In this section, we will consider SDEs of the form
\begin{align}\label{eqn:integral_eqn_add_noise}
 Y_t = \xi + \int_0^t b(Y_s)\,ds + \sum_{i = 1}^m \int_0^t \sigma_i(s)\, dX_s^i 
\end{align}
with $\xi \in \R^d$. Here, $b \colon \R^d \to \R^d$ is a continuous vector field, $\sigma_1, \ldots, \sigma_m \colon [0,T] \to \R^d$ are continuous functions and $X = (X^1,\ldots, X^m) \colon [0,T] \to \R^m$ is a Gaussian process with continuous trajectories. The stochastic integrals in \eqref{eqn:integral_eqn_add_noise} can either be defined pathwise (e.g. as Young integrals, cf. \cite{You36} or \cite[Section 6]{FV10}) or by probabilistic means (e.g. as Wiener integrals). At this stage, we only assume that the stochastic integrals are defined in such a way that they introduce a bounded linear map from the Gaussian space $C([0,T],\R^m)$ to the space $C([0,T],\R^d)$ which implies that the sum of the integrals is again a Gaussian processes. Therefore, there is no loss of generality to consider SDEs of the form
\begin{align}\label{eq:sde_additive_noise}
 Y_t = \xi + \int_0^t b(Y_s)\,ds + X_t 
\end{align}
instead of \eqref{eqn:integral_eqn_add_noise} where $X \colon [0,T] \to \R^d$ is an $\R^d$-valued Gaussian process starting at $0$ with continuous sample paths. Under mild regularity assumptions on $b$ (e.g. continuous, locally Lipschitz continuous and linear growth), the equation \eqref{eq:sde_additive_noise} can be solved pathwise for every continuous trajectory of the Gaussian process. We aim to establish concentration inequalities for the law of the solution $Y \colon [0,T] \to \R^d$. Our strategy will be to show that the solution map $X(\omega) \mapsto Y(\omega)$ is Lipschitz continuous, which implies the concentration inequality by the contraction principle stated in Lemma \ref{lemma:transp_ineq_trans_under_loc_lip_maps}. 

In the case of $X$ being a Wiener process, Djellout, Guillin and Wu show in \cite[Proposition 5.4]{DGW04} that the quadratic transportation inequality holds even for the metric $d_{\mathcal{H}}$. However, their analysis relies on the fact that in case of the Wiener process, the Cameron Martin space is explicitly known; it is the Sobolev space $H^1_0 = W_0^{1,2}$. For a general Gaussian process, the Cameron Martin space is usually only implicitly defined, and showing Lipschitz continuity for the corresponding metric is not obvious. On the other hand, there are often continuous embeddings available for the Cameron Martin space into the space of paths with finite $p$-variation. Showing Lipschitz continuity for the $p$-variation metric is a much easier task which will immediately yield concentration inequalities in $p$-variation topology.

We start with a simple calculation.

\begin{proposition}\label{prop:lip_qvar}
 Let $x_1, x_2 \colon [0,T] \to \R^d$ be two continuous paths and choose $\xi^1, \xi^2 \in \R^d$. Consider the equations
 \begin{align}\label{eqn:add_noise_determ_Lip}
  y^i_t = \xi^i + \int_0^t b(y^i_s)\, ds + x^i_t; \quad t \in [0,T],\ i = 1,2
 \end{align}
 where $b \colon \R^d \to \R^d$ is Lipschitz continuous with Lipschitz constant $L$. 
 
 Then the equations \eqref{eqn:add_noise_determ_Lip} have unique, continuous solutions $y^1, y^2 \colon [0,T] \to \R^d$ and the estimate
 \begin{align}\label{eqn:add_noise_eq_lip-q-var}
  \|y^1 - y^2\|_{q-\text{var}} \leq 2^{1 - \frac{1}{q}} \exp\left(2^{1 - \frac{1}{q}} L T\right) \left( L T |\xi^1 - \xi^2| + \|x^1 - x^2\|_{q-\text{var}} \right)
 \end{align}
 holds for every $q \in [1,\infty)$.

\end{proposition}

\begin{proof}
 Existence and uniqueness is classical, we only need to prove the estimate \eqref{eqn:add_noise_eq_lip-q-var}. Fix some $t \in [0,T]$. Let $(t_i)$ be a partition of $[0,t]$. Using the equations shows that
 \begin{align*}
  |y^1_{t_{i+1}} - y^1_{t_i} - y^2_{t_{i+1}} + y^2_{t_i}| \leq L \int_{t_i}^{t_{i+1}} |y^1_s - y^2_s|\, ds + |x^1_{t_{i+1}} - x^1_{t_i} - x^2_{t_{i+1}} + x^2_{t_i}|
 \end{align*}
 for every $t_i < t_{i+1}$. Taking both sides to the power $q$ and summing over all increments gives
 \begin{align*}
  \sum_{t_i} |y^1_{t_{i+1}} - y^1_{t_i} - y^2_{t_{i+1}} + y^2_{t_i}|^q \leq 2^{q-1} L^q \left( \int_0^t |y^1_s - y^2_s|\, ds \right)^q + 2^{q-1} \sum_{t_i} |x^1_{t_{i+1}} - x^1_{t_i} - x^2_{t_{i+1}} + x^2_{t_i}|^q.
 \end{align*}
 Taking now the supremum over all partitions implies
 \begin{align*}
  \|y^1 - y^2\|_{q-\text{var};[0,t]} \leq 2^{1-\frac{1}{q}} L \int_0^t |y^1_s - y^2_s|\, ds + 2^{1-\frac{1}{q}} \|x^1 - x^2\|_{q-\text{var};[0,t]}.
 \end{align*}
 The integral can be estimated by
 \begin{align*}
  \int_0^t |y^1_s - y^2_s|\, ds \leq \int_0^t \|y^1 - y^2\|_{q-\text{var};[0,s]} \, ds + t|\xi^1 - \xi^2|. 
 \end{align*}
 Gronwall's inequality implies the claim.

\end{proof}

  Set $C_{\xi} = C_{\xi}([0,T];\R^d)$.

\begin{theorem}\label{thm:transp_cost_ineq_add_noise_case}
	Assume that $b \colon \R^d \to \R^d$ is Lipschitz continuous with Lipschitz constant $L > 0$. Let $X \colon [0,T] \to \R^d$ be a continuous Gaussian process with Cameron Martin space $\mathcal{H}$, and
	assume that there is a continuous embedding 
	\begin{align}\label{eqn:CM_embedding_q-var_add_noise}
	    \iota \colon \mathcal{H} \hookrightarrow C^{q-\text{var}}
	 \end{align}
	for some $q \in [1,\infty)$. Let $Y$ be the solution to the SDE \eqref{eq:sde_additive_noise} and let $\mu$ be the law of $Y$. 
	
	Then for every $\nu \in P(C_{\xi})$,
	  \begin{align}\label{eqn:conc_ineq_q-var_add_noise}
	    \inf_{\pi \in \Pi(\nu,\mu)} \int_{C_{\xi} \times C_{\xi}} d_{q-\text{var}}(x,y)^2\, d\pi(x,y) \leq 2^{3 - \frac{2}{q}} \exp\left(2^{2 - \frac{1}{q}} L T \right) \|\iota\|_{\mathcal{H} \hookrightarrow C^{q-\text{var}}}^2 \, H(\nu\,|\,\mu).
	  \end{align}
\end{theorem}

\begin{proof}
	Follows from Theorem \ref{thm:talagrand_strong_form_linear}, the contraction principle in Lemma \ref{lemma:transp_ineq_trans_under_loc_lip_maps} and Proposition \ref{prop:lip_qvar}.
\end{proof}

\begin{remark}
  Embeddings of the form \eqref{eqn:CM_embedding_q-var_add_noise} play a crucial role in Gaussian rough paths theory and we will revisit them also in the next section. Sufficient conditions for such embeddings, as well as many examples of Gaussian processes for which they hold, are given in \cite{FGGR16}. 

\end{remark}

Next, we aim to relax the assumptions on $b \colon \R^d \to \R^d$. In case of the Brownian motion, it is well known (cf. \cite{PR07}) that \eqref{eq:sde_additive_noise} has a unique solution provided $b$ is continuous and satisfies the following \emph{one sided Lipschitz condition}:
    \begin{itemize}
     \item[(i)] There exists a constant $C_1$ such that
     \begin{align}\label{eqn:bound_radial_growth_uniq}
      \langle b(\xi) - b(\zeta), \xi - \zeta \rangle \leq C_1|\xi - \zeta|^2 \quad \text{for every } \xi, \zeta \in \R^d.
     \end{align}
     \end{itemize}
However, one has to be careful when solving \eqref{eq:sde_additive_noise} pathwise: In \cite[p. 43]{CHJ13}, the authors show that there are trajectories which lead to explosion in finite time of solutions to \eqref{eq:sde_additive_noise} although the vector field $b$ satisfies \eqref{eqn:bound_radial_growth_uniq}. In \cite{RS16} and \cite{SS16}, a further condition on $b$ was introduced. Together with \eqref{eqn:bound_radial_growth_uniq}, this condition prevents explosion, even in the more general case of multiplicative noise. This condition takes the following form:
    \begin{itemize}
     \item[(ii)] There exists a constant $C_2$ such that
     \begin{align}\label{eqn:bound_tang_growth_uniq}
      \left| b(\xi) - b(\zeta) - \frac{\langle b(\xi) - b(\zeta) , \xi - \zeta \rangle (\xi - \zeta)}{| \xi - \zeta |^2} \right| \leq C_2|\xi - \zeta| \quad \text{for every } \xi, \zeta \in 
      \R^d \text{ with } \xi - \zeta \neq 0.
     \end{align}
    \end{itemize}

In the following, we will assume both \eqref{eqn:bound_radial_growth_uniq} and \eqref{eqn:bound_tang_growth_uniq}.

Let $x \colon [0,T] \to \R^d$ be continuous. It is shown in \cite[Lemma 4.1 and Lemma 4.2]{RS16} that for $b$ continuous and satisfying \eqref{eqn:bound_radial_growth_uniq} and \eqref{eqn:bound_tang_growth_uniq}, the equation
\begin{align*}
 \dot{z}_t = b(z_t + x_t)
\end{align*}
generates a continuous two-parameter flow.

\begin{lemma}\label{lemma:one_sided_lip_est}
 Let $x^1, x^2 \colon [S,T] \to \R^d$ be continuous. Consider the solutions $z^1, z^2 \colon [S,T] \to \R^d$ to the equations
 \begin{align*}
  \dot{z}^i_t &= b(z^i_t + x^i_t); \quad t \in [S,T] \\
  z^i_S &= \xi^i
 \end{align*}
 for $i = 1,2$ and two initial conditions $\xi^1, \xi^2 \in \R^d$. Assume that 
 \begin{align*}
  |z^1_t - z^2_t| \geq |x^1_t - x^2_t|\quad \text{for all } t\in [S,T].
 \end{align*}
 Then
 \begin{align*}
  \sup_{t \in [S,T]} |z^1_t - z^2_t| \leq e^{\frac{3}{2}|T-S|(C_1 + C_2)} \left( |\xi^1 - \xi^2| + \sqrt{(C_1 + C_2)|T - S|} \sup_{t \in [S,T]} |x^1_t - x^2_t| \right).
 \end{align*}

\end{lemma}

\begin{proof}
 For all $t \in [S,T]$,
 \begin{align*}
  |z^1_t - z^2_t|^2 = |\xi^1 - \xi^2|^2 + 2 \int_S^t \langle z^1_s - z^2_s, b(z_s^1 + x_s^1) - b(z_s^2 + x_s^2) \rangle\, ds .
 \end{align*}
 Fix $s \in [S,t]$. Choose $\alpha, \beta \in \R$ such that
 \begin{align*}
  b(z_s^1 + x_s^1) - b(z_s^2 + x_s^2) = \alpha (z_s^1 + x_s^1 - z_s^2 - x_s^2) + \beta v
 \end{align*}
 where $v \perp (z_s^1 + x_s^1 - z_s^2 - x_s^2)$ in the case $z_s^1 + x_s^1 - z_s^2 - x_s^2 \neq 0$ and $v$ arbitrary but $\beta = 0$ otherwise. By \eqref{eqn:bound_radial_growth_uniq}, $\alpha \leq C_1$ and \eqref{eqn:bound_tang_growth_uniq} implies that
 \begin{align*}
  |\beta| \leq C_2 |z_s^1 + x_s^1 - z_s^2 - x_s^2|.
 \end{align*}
 Note that
 \begin{align*}
  \langle z^1_s - z^2_s, z_s^1 + x_s^1 - z_s^2 - x_s^2 \rangle \geq |z^1_s - z^2_s|^2 - |z^1_s - z^2_s| |x^1_s - x^2_s| \geq 0
 \end{align*}
 by assumption, therefore
 \begin{align*}
  \alpha \langle z^1_s - z^2_s, z_s^1 + x_s^1 - z_s^2 - x_s^2 \rangle \leq \frac{C_1}{2}(3 |z^1_s - z^2_s|^2 + |x^1_s - x^2_s|^2).
 \end{align*}
 Furthermore,
 \begin{align*}
  \beta \langle z^1_s - z^2_s, v \rangle \leq \frac{C_2}{2}(3 |z^1_s - z^2_s|^2 + |x^1_s - x^2_s|^2).
 \end{align*}
 This implies that
 \begin{align*}
  |z^1_t - z^2_t|^2 \leq |\xi^1 - \xi^2|^2 + (C_1 + C_2)|T - S| \sup_{u \in [S,T]} |x^1_u - x^2_u|^2 + 3(C_1 + C_2) \int_S^t |z^1_s - z^2_s|^2\, ds
 \end{align*}
 holds for all $t \in [S,T]$. Gronwall's Lemma gives the claim.

\end{proof}

\begin{proposition}\label{prop:lip_cont_one_sided_cond}
 Let $x_1, x_2 \colon [0,T] \to \R^d$ be two continuous paths starting at $0$ and choose $\xi^1, \xi^2 \in \R^d$. Consider the equations
 \begin{align}\label{eqn:add_noise_determ}
  y^i_t = \xi^i + \int_0^t b(y^i_s)\, ds + x^i_t; \quad t \in [0,T],\ i = 1,2
 \end{align}
 where $b \colon \R^d \to \R^d$ is continuous and satisfies \eqref{eqn:bound_radial_growth_uniq} and \eqref{eqn:bound_tang_growth_uniq}.
 
 Then the equations \eqref{eqn:add_noise_determ} have unique, continuous solutions $y^1, y^2 \colon [0,T] \to \R^d$ and the estimate
 \begin{align}\label{eqn:add_noise_eq_lip}
  \|y^1 - y^2\|_{\infty} \leq  3 e^{2T(C_1 + C_2)} (|\xi^1 - \xi^2| + \|x^1 - x^2\|_{\infty})
 \end{align}
 holds.

\end{proposition}

\begin{proof}
 The fact that the equations \eqref{eqn:add_noise_determ} possess unique solutions is a special case of \cite[Theorem 4.3]{RS16}. We only need to prove the estimate \eqref{eqn:add_noise_eq_lip}. It is easy to see that if $z^i \colon [0,T] \to \R^d$, $i = 1,2$ denote the solutions to
 \begin{align}
  \begin{split}
   \dot{z}^i_t &= b(z^i_t + x^i_t) \\
    z^i_0 &= \xi^i, 
  \end{split}
 \end{align}
 the solutions $y^i$ to \eqref{eqn:add_noise_determ} are given by $z^i + x^i$. Set 
 \begin{align*}
  \beta_t := \left| |z^1_t - z^2_t| - |x^1_t - x^2_t| \right|.
 \end{align*}
 Let $\delta > 0$. We define a sequence of increasing numbers $0 =: \tau_0 < \tau_1 < \ldots$ as follows:
 \begin{align*}
  \tau_1 &:= \inf_{t \geq \tau_0} \{\beta_t > \delta \} \wedge T \\
  \tau_2 &:= \inf_{t \geq \tau_1} \{\beta_t < \delta/2 \} \wedge T \\
  \tau_3 &:= \inf_{t \geq \tau_2} \{\beta_t > \delta \} \wedge T \\
  &\vdots
 \end{align*}
  Note that there is a minimal number $N \in \N$ such that $\tau_N = T$. Indeed, otherwise we constructed an increasing sequence $(\tau_n)$, bounded by $T$, which therefore converges towards some number $\tau$, but $(\beta_{\tau_n})$ can clearly not converge although it is continuous, which is a contradiction. By construction, for every $n = 0,\ldots,N-1$, one either has $\beta_t \leq \delta$ or $\beta_t \geq \delta/2$ for every $t \in [\tau_n,\tau_{n+1}]$. In the first case,
  \begin{align*}
   |z^1_t - z^2_t| \leq \delta + \sup_{u \in [\tau_n,\tau_{n+1}]} |x^1_u - x^2_u|
  \end{align*}
  for every $t \in [\tau_n,\tau_{n+1}]$. In the case $\beta_t \geq \delta/2$, we either have $|x^1_t - x^2_t| \geq \delta/2 + |z^1_t - z^2_t| \geq |z^1_t - z^2_t|$ which implies
  \begin{align*}
   |z^1_t - z^2_t| \leq  \sup_{u \in [\tau_n,\tau_{n+1}]} |x^1_u - x^2_u|
  \end{align*}
  for all $t \in [\tau_n,\tau_{n+1}]$, or $|x^1_t - x^2_t| \leq |z^1_t - z^2_t|$ for all $t \in [\tau_n,\tau_{n+1}]$. In the second case, we can use Lemma \ref{lemma:one_sided_lip_est} to obtain the estimate
  \begin{align*}
   |z^1_t - z^2_t| \leq e^{2(\tau_{n+1} - \tau_n)(C_1 + C_2)} \left( |z^1_{\tau_n} - z^2_{\tau_n}| + \sup_{u \in [\tau_n,\tau_{n+1}]} |x^1_u - x^2_u| \right)
  \end{align*}
  which holds for all $t \in [\tau_n,\tau_{n+1}]$. In the case $n = 0$, we have $|z^1_{\tau_n} - z^2_{\tau_n}| = |\xi^1 - \xi^2|$. For $n \geq 1$, we know that $\beta_t \leq \delta$ for $t \in [\tau_{n-1},\tau_{n}]$, therefore
  \begin{align*}
   |z^1_{\tau_n} - z^2_{\tau_n}| \leq  \sup_{u \in [\tau_{n-1},\tau_{n}]} |z^1_u - z^2_u| \leq \delta + \sup_{u \in [\tau_{n-1},\tau_{n}]} |x^1_u - x^2_u|.
  \end{align*}
  This shows that in all cases we have considered, the estimate
  \begin{align*}
   |z^1_t - z^2_t| \leq 2 e^{2T(C_1 + C_2)} \left( \delta + |\xi^1 - \xi^2| + \|x^1 - x^2\|_{\infty} \right)
  \end{align*}
  holds true which implies that
  \begin{align*}
   \|z^1 - z^2\|_{\infty} \leq 2 e^{2T(C_1 + C_2)} \left( \delta + |\xi^1 - \xi^2| + \|x^1 - x^2\|_{\infty} \right).
  \end{align*}
  Note that this is true for any $\delta > 0$, therefore we can conclude that
  \begin{align*}
   \|z^1 - z^2\|_{\infty} \leq 2 e^{2T(C_1 + C_2)} \left(|\xi^1 - \xi^2| + \|x^1 - x^2\|_{\infty} \right)
  \end{align*}
  holds true. The claim follows from the equality $y^i = z^i + x^i$ and the triangle inequality.

\end{proof}

\begin{theorem}\label{thm:transp_cost_ineq_add_noise_one_sided_lip}
	Assume $b \colon \R^d \to \R^d$ is continuous and satisfies \eqref{eqn:bound_radial_growth_uniq} and \eqref{eqn:bound_tang_growth_uniq}. Let $X \colon [0,T] \to \R^d$ be a continuous Gaussian process with corresponding Gaussian measure $\gamma$ on the space of continuous functions, and let $\sigma^2$ be defined as in \eqref{eqn:var_gauss_measure}. Let $Y$ be the solution to the SDE \eqref{eq:sde_additive_noise} and let $\mu$ be the law of $Y$. 
	
	Then for every $\nu \in P(C_{\xi})$,
	  \begin{align*}
	    \inf_{\pi \in \Pi(\nu,\mu)} \int_{C_{\xi} \times C_{\xi}} \| x - y\|^2_{\infty} \, d\pi(x,y) \leq 18 \sigma^2 e^{4T(C_1 + C_2)} \, H(\nu\,|\,\mu).
	  \end{align*}
\end{theorem}

\begin{proof}
	This is a consequence of Corollary \ref{corollary:weak_conclusion}, the contraction principle in Lemma \ref{lemma:transp_ineq_trans_under_loc_lip_maps} and Proposition \ref{prop:lip_cont_one_sided_cond}.
\end{proof}

\begin{example}\label{ex:bifBm}
 We finally discuss an example to illustrate our findings. Let $B^{H,K} \colon [0,T] \to \R^m$ be a \emph{bifractional Brownian motion}, i.e. a continuous, centered Gaussian process with independent components and the covariance of each component is given by
 \begin{align*}
  R(s,t) = \frac{1}{2^K} \left( (s^{2H} + t^{2H})^K - |t - s|^{2HK} \right)
 \end{align*}
 with $H \in (0,1)$ and $K \in (0,1]$. This process was introduced in \cite{HV03} and further studied e.g. in \cite{RT06,KRT07}. Note that for $K = 1$, we obtain a fractional Brownian motion, and for $K = 1$ and $H = 1/2$ we have the usual Brownian motion. In the general form, it is not known whether the process can be written as a stochastic integral with respect to Brownian motion (as for the fractional Brownian motion) or whether it is adapted to a Brownian filtration. This rules out any Girsanov transformation techniques. It can be shown that $B^{H,K}$ has sample paths of $\alpha$-H\"older regularity for any $\alpha < HK$ (and the sample paths are therefore of finite $1/ \alpha$-variation), but not better. In \cite[Example 2.12]{FGGR16}, it was shown that the corresponding Cameron Martin space can be continuously embedded in the space of paths with finite $q$-variation for $q = (HK + 1/2)^{-1} \vee 1$. We aim to take $B^{H,K}$ as the driver in equation \eqref{eqn:integral_eqn_add_noise}. If all $\sigma_
i$ have finite $p$-variation for $p < (1 - HK)^{-1}$, we can define the stochastic integrals pathwise as Young integrals. We can relax the assumptions on $\sigma$ if we are only interested in defining the stochastic integral by stochastic means. More precisely, if $HK \geq 1/2$ and if all $\sigma_i$ are continuous, we can integrate each $\sigma_i$ against any Cameron Martin path, using Riemann-Stieltjes integrals. If $HK \leq 1/2$ and if all $\sigma_i$ have finite $p$-variation for some $p \geq 1$ satisfying
 \begin{align*}
  \frac{1}{p} > \frac{1}{2} - HK,
 \end{align*}
 we can integrate each $\sigma_i$ against any Cameron Martin path using Young integrals. In these two cases, integration induces a bounded linear map from the associated Cameron Martin space $\mathcal{H}$ to the space $C([0,T],\R^d)$, and therefore also to the Hilbert space $L^2([0,T],\R^d)$. If $(e_n)$ denotes an orthonormal basis of $L^2$, we define a scalar product
 \begin{align*}
  \langle x, y \rangle^{\sim} := \sum_n \frac{1}{n^2} \langle x, e_n \rangle_{L^2} \langle y, e_n \rangle_{L^2}.
 \end{align*}
 Let $\tilde{L}^2$ denote the space $L^2([0,T],\R^d)$ equipped with this scalar product. Then integration induces a Hilbert-Schmidt operator from $\mathcal{H}$ to $\tilde{L}^2$, which can therefore be uniquely extended to the whole space $C([0,T],\R^d)$ almost surely and induces a Gaussian measure on $\tilde{L}^2$ (cf. e.g. \cite[Theorem 3.44]{Hai09}). It can be shown (using the explicit bounds of this map) that the associated Gaussian process on $\tilde{L}^2$ has actually continuous sample paths almost surely, and that its Cameron Martin space can again be continuously embedded in the space of $q$-variation paths with the same choice of $q$. From now on, assume that the $\sigma_i$ satisfy one of the stated regularity assumptions. In case that the drift $b$ is Lipschitz continuous, we can solve \eqref{eqn:integral_eqn_add_noise}, and the law $\mu$ of the solution $Y$ satisfies the quadratic transport inequality
	  \begin{align}\label{eqn:transp_bifrac_lip}
	    \inf_{\pi \in \Pi(\nu,\mu)} \int_{C_{\xi} \times C_{\xi}} d_{q-\text{var}}(x,y)^2\, d\pi(x,y) \leq C \, H(\nu\,|\,\mu)
	  \end{align}
for some constant $C > 0$ and any $\nu \in P(C_{\xi})$ by Theorem \ref{thm:transp_cost_ineq_add_noise_case}. Note that $q < 1/(HK)$ (e.g. $q = 1$ in case of the Brownian motion), and we cannot expect that the sample paths of $Y$ itself have finite $q$-variation. Assuming only continuity, \eqref{eqn:bound_radial_growth_uniq} and \eqref{eqn:bound_tang_growth_uniq} for $b$, we can still solve \eqref{eqn:integral_eqn_add_noise}, and the law $\mu$ of the solution $Y$ satisfies the quadratic transport inequality
	  \begin{align}\label{eqn:transp_bifrac_one_sided_lip}
	    \inf_{\pi \in \Pi(\nu,\mu)} \int_{C_{\xi} \times C_{\xi}} \| x - y\|_{\infty}^2 \, d\pi(x,y) \leq C \, H(\nu\,|\,\mu)
	  \end{align}
for another constant $C > 0$ and any $\nu \in P(C_{\xi})$ by Theorem \ref{thm:transp_cost_ineq_add_noise_one_sided_lip}. 

In case of the fractional Brownian motion (i.e. $K = 1$), the transport inequalitities \eqref{eqn:transp_bifrac_lip} and \eqref{eqn:transp_bifrac_one_sided_lip} may be compared to the corresponding results obtained in \cite{Sau12} (namely Theorem 1 and Theorem 3). Note that our results imply those and are even stronger in several regards (quadratic transport inequality instead of simple one, larger metric, less regularity assumptions on the vector fields).
\end{example}

\subsection{SDEs with multiplicative noise}\label{sec:sde_mult_noise}

Next we will consider SDEs with multiplicative noise, i.e. equations of the form

\begin{align}\label{eqn:Strat_SDE_mult_noise_case}
  Y_t = \xi + \int_0^t b(Y_s)\,ds + \sum_{i = 1}^m \int_0^t \sigma_i(Y_s)\,\circ dX^i_s
\end{align}
where $\xi \in \R^d$, $X = (X^1,\ldots,X^m)$ is a continuous $m$-dimensional Gaussian process and $b,\sigma_1,\ldots \sigma_m \colon \R^d \to \R^d$ are continuous vector fields. The problem in \eqref{eqn:Strat_SDE_mult_noise_case} is of course to make sense of the stochastic integrals if $X$ is not a martingale.

We start to discuss a simple case; namely, we assume that the driving process is one dimensional. Under further assumptions on the vector fields, we can use the Doss-Sussmann representation to define the solution to \eqref{eqn:Strat_SDE_mult_noise_case} pathwise for any continuous driving signal\footnote{Note that we can also use rough paths theory for $m = 1$ to make sense of \eqref{eqn:Strat_SDE_mult_noise_case} since the iterated integrals are canonically given as products in this case.}. If we further assume that also the solution space is one dimensional we can follow \cite{Sau12} to derive the following result:

\begin{theorem}
 Assume $m = d = 1$ and consider the equation
\begin{align}\label{eqn:SDE_mult_noise_case_1d}
  Y_t = \xi + \int_0^t b(Y_s)\,ds + \int_0^t \sigma(Y_s)\,\circ dX_s
\end{align}
where $X \colon [0,T] \to \R$ is a continuous Gaussian process. We further assume that $b$ is bounded by some constant $B$ and Lipschitz continuous with Lipschitz constant $L_b$. For the diffusion vector field $\sigma$, we assume that it is Lipschitz continuous with Lipschitz constant $L_{\sigma}$ and that there are constants $0 < \sigma_1 \leq \sigma_2$ such that $\sigma_1 \leq \sigma(x) \leq \sigma_2$ for all $x \in \R$.

Then the equation \eqref{eqn:SDE_mult_noise_case_1d} has a unique continuous solution $Y$ and its law $\mu$ satisfies the quadratic transport inequality
	  \begin{align*}
	    \inf_{\pi \in \Pi(\nu,\mu)} \int_{C_{\xi} \times C_{\xi}} \| x - y\|_{\infty}^2 \, d\pi(x,y) \leq C \, H(\nu\,|\,\mu)
	  \end{align*}
for all $\nu \in P(C_{\xi})$ where $C > 0$ is a constant depending on the variance of the Gaussian measure given in \eqref{eqn:var_gauss_measure}, $T$ and all constants above.

\end{theorem}

\begin{proof}
 Under the stated conditions, one can show, using the Lamperti transform (cf. \cite[proof of Theorem 12 on p. 12]{Sau12}) that the solution map associated to \eqref{eqn:Strat_SDE_mult_noise_case} is Lipschitz continuous on the space of continuous functions. The result follows from Corollary \ref{corollary:weak_conclusion} and the contraction principle in Lemma \ref{lemma:transp_ineq_trans_under_loc_lip_maps}.
\end{proof}

Note that this result generalizes \cite[Theorem 2]{Sau12} to arbitrary Gaussian processes. It is even stronger than \cite[Theorem 2]{Sau12} since we can deduce the quadratic transportation inequality, not only the simple one. We can also deduce the quadratic transportation inequality for general Gaussian processes under the conditions stated in \cite[Theorem 4]{Sau12} for the uniform metric. Indeed, an inspection of the proof reveals that under these conditions, the solution map associated to \eqref{eqn:Strat_SDE_mult_noise_case} is again Lipschitz, therefore we can conclude as before.

Now assume that $X$ is an $m$-dimensional Brownian motion. In contrast to the additive noise case or the one dimensional case, the solution map $I(\cdot,\xi) \colon C_0([0,T],\R^m) \to C_{\xi}([0,T],\R^d)$ which assigns to each Brownian path the solution path to the SDE \eqref{eqn:Strat_SDE_mult_noise_case} will in general \emph{not} be \mbox{(Lipschitz-) continuous}. This issue can be overcome using Lyons' rough paths theory. Indeed, rough paths theory shows that there is a Polish space $\mathcal{D}^{0,p}_g$ (cf. \cite[Definition 9.15 and Proposition 8.25]{FV10} for the precise definition) such that the diagram
\begin{align}\label{eqn:comm_diag}
\begin{tikzpicture}[node distance=2cm, auto]
  \node (C) {$\mathcal{D}^{0,p}_g$};
  \node (P) [below of=C] {$C_0$};
  \node (Ai) [right of=P] {$C_{\xi}$};
  \draw[->] (C) to node {$\mathbf{I}(\cdot,\xi)$} (Ai);
  \draw[->] (P) to node [swap] {$S$} (C);
  \draw[->] (P) to node [swap] {$I(\cdot,\xi)$} (Ai);
\end{tikzpicture}
\end{align}
commutes almost surely and the map $\mathbf{I}(\cdot,\xi) \colon \mathcal{D}^{0,p}_g \to C_{\xi}$ is locally Lipschitz continuous. The map $S \colon C_0 \to \mathcal{D}^{0,p}_g$ is constructed w.r.t. the Wiener measure on the path space $C_0$. Using a pathwise approach, one is not restricted to Wiener measure and it is indeed possible to construct lift maps $S$ w.r.t. more general Gaussian measures $\gamma$ (cf. \cite{CQ02}, \cite{FV10-2}). In this case, one \emph{defines} $I(\cdot,\xi) := \mathbf{I}(S(\cdot),\xi)$ which gives rise to solutions of SDEs of the form
\begin{align}\label{eqn:mult_case_gaussian_sde}
  Y_t = \xi + \int_0^t b(Y_t) + \sum_{i = 1}^m \int_0^t \sigma_i(Y_s)\,\circ dX^i_s 
\end{align}
where $X = (X^1,\ldots,X^m)$ is the canonical process induced by the Gaussian measure $\gamma$.
Our key result will be that for \textit{Brownian-like} Gaussian processes (we will be more precise later), we have an estimate of the form
\begin{align*}
 \| I(x,\xi) - I(y,\xi) \|_{p-\text{var}} \leq L(y) d_{\mathcal{H}}(x,y)
\end{align*}
almost surely for every $x,y \in C_0$ where $L$ is a random variable which possesses every moment w.r.t. the Gaussian measure $\gamma$. Together with Lemma \ref{lemma:transp_ineq_trans_under_loc_lip_maps}, this yields a transportation inequality which is stated in Theorem \ref{thm:main_result_multiplicative_sdes}.

We will not make an attempt to give an overview to rough paths theory since we will use it merely as a tool. Instead, we refer to the monographs \cite{LQ02}, \cite{LCL07}, \cite{FV10} and \cite{FH14}. The terms and notation we are using coincides with the one from \cite{FV10} with the only exception that we use the symbol $\mathcal{D}_g^{0,p}$ to denote the space of geometric $p$-variation rough paths $C^{0,p-\text{var}}_0([0,T];G^{[p]}(\R^m))$ equipped with the $p$-variation metric. By \cite[Proposition 8.25]{FV10}, this space is Polish. 

We start with some deterministic estimates for rough paths. If $\omega$ is a control function and $\alpha > 0$, recall the definition of $N_{\alpha}(\omega;[s,t])$ resp. of $N_{\alpha}(\mathbf{x};[s,t])$ for geometric rough paths $\mathbf{x}$ (\cite{CLL13}, \cite{FR12b}). The next proposition is a version of \cite[Theorem 4]{BFRS16} for the $p$-variation metric.

\begin{proposition}
\label{prop:loc_lip_integrability_p-var_top} Let $\mathbf{x}^1$ and $\mathbf{x}^2$ be weakly geometric $p$-rough paths for some $p \geq 1$. Consider the rough differential equations (RDEs)
\begin{equation*}
dy_{t}^{j} = \sigma^{j}(y_{t}^{j})\,d\mathbf{x}_{t}^{j};\ y_{S}^{j}\in \mathbb{R%
}^{d}
\end{equation*}%
for $j=1,2$ on some interval $[S,T]$ where $f^{1} = (f^1_i)_{i = 1,\ldots,m}$ and $f^{2} = (f^2_i)_{i=1,\ldots,m}$ are two families of vector
fields in $\R^d$, $\theta > p$ and $\beta $ is a bound on\footnote{We mean Lipschitz it the sense of Stein , cf. \cite[Chapter 10]{FV10}} $|\sigma^{1}|_{Lip^{\theta }}$ and $%
|\sigma^{2}|_{Lip^{\theta }}$. 

Then for every $\alpha >0$ there is a constant $%
C=C(\theta ,p,\beta ,\alpha )$ such that%
\begin{eqnarray*}
d_{p-\text{var};[S,T]}( y^{1},y^{2}) &\leq
&\ C\left[ |y_{S}^{1}-y_{S}^{2}|+\left\vert \sigma^{1} - \sigma^{2}\right\vert _{\text{%
Lip}^{\theta -1}}+\rho _{p-\text{var};[S,T]}(\mathbf{x}^{1},\mathbf{x}^{2})%
\right] \\
&&\times ( \|\mathbf{x}^1 \|_{p-\text{var};[S,T]} + \|\mathbf{x}^2 \|_{p-\text{var};[S,T]} + 1 ) \\
&&\times \exp \left\{ C\left( N_{\alpha}(\mathbf{x}^{1};[S,T])+N_{\alpha
}(\mathbf{x}^{2};[S,T])+1\right) \right\}
\end{eqnarray*}%
holds.
\end{proposition}

\begin{proof}
 The proof follows \cite[Lemma 7 and Theorem 4]{BFRS16}. Let $\omega$ be a control function such that $\sup_{s<t} \frac{\|\mathbf{x}^j \|}{\omega(s,t)^{1/p}} \leq 1$ for $j = 1,2$. Set $\bar{y} := y^1 - y^2$ and
 \begin{align*}
  \kappa :=\frac{\left\vert \sigma^{1} - \sigma^{2}\right\vert _{\text{Lip}^{\theta -1}}}{%
\beta }+\rho _{p-\omega ;[S,T]}(\mathbf{x}^{1},\mathbf{x}^{2}).
 \end{align*}
 We claim that there is a constant $C = C(\theta,p)$ such that for every $s<t$,
 \begin{align}\label{eqn:claim_1_prop_loc_lip_p-var}
  \|\bar{y}\|_{p-\text{var};[s,t]} \leq C \beta \omega(s,t)^{\frac{1}{p}} \left(\|\bar{y}\|_{\infty;[s,t]} + \kappa \right) \exp \left\{ C \beta^p(N_{\alpha}(\omega;[s,t]) + 1) \right\}.
 \end{align}
 Indeed, as it was shown in the proof of \cite[Lemma 7]{BFRS16},
 \begin{align*}
  2^{1-p} |\bar{y}_{s,v}|^p \leq C \beta^p \omega(u,v) (|\bar{y}_u| + \kappa)^p  \exp \left\{ C \beta^p \omega(u,v) \right\} + |\bar{y}_{s,u}|^p
 \end{align*}
 for every $[u,v] \subseteq [s,t]$. Thus if $s = \tau_0 < \ldots < \tau_M < \tau_{M+1} = v$,
 \begin{align*}
  |\bar{y}_{s,v}|^p \leq 2^{(M+1)(p-1)} C \beta^p \omega(s,v) (\|\bar{y}\|_{\infty;[s,v]} + \kappa)^p  \exp \left\{ C \beta^p \sum_{i = 0}^M \omega(\tau_i,\tau_{i+1}) \right\}
 \end{align*}
 for every $s \leq v \leq t$. Choosing $\tau_0 = s$, $\tau_{i+1} = \inf_t\{\omega(\tau_i,t) \geq \alpha \} \wedge v $ gives 
  \begin{align*}
  |\bar{y}_{s,v}|^p \leq C \beta^p \omega(s,v) (\|\bar{y}\|_{\infty;[s,v]} + \kappa)^p  \exp \left\{ C \beta^p (N_{\alpha}(\omega;[s,v]) + 1) \right\}
 \end{align*}
  for every $s \leq v \leq t$ and \eqref{eqn:claim_1_prop_loc_lip_p-var} follows. Now we can use the conclusion from \cite[Lemma 7]{BFRS16} to see that
  \begin{align*}
   \|\bar{y}\|_{p-\text{var};[s,t]} \leq C \beta \omega(s,t)^{\frac{1}{p}} \left(\|\bar{y}_s\| + \kappa \right) \exp \left\{ C \beta^p(N_{\alpha}(\omega;[s,t]) + 1) \right\}
  \end{align*}
  holds for every $s<t$. We conclude as in \cite[Theorem 4]{BFRS16}.

\end{proof}

\begin{lemma}\label{lemma:estimates_diff_young_shift_inhom_rp_metric}
Let $\mathbf{x}$ be a weakly geometric $p$-rough path with $p \in [2,3)$ and $h$ a path of finite $q$-variation with $1 \leq q \leq p$ and $\frac{1}{p} + \frac{1}{q} > 1$. Then there is a constant $C = C(p,q)$ such that
  \begin{align*}
   \rho_{p-\text{var};[S,T]}(T_h(\mathbf{x}),\mathbf{x}) \leq C_{p,q}(1 \vee \|x\|_{p-\text{var};[S,T]})(\| h\|_{q-\text{var};[S,T]} + \| h\|_{q-\text{var};[S,T]}^2)
  \end{align*}
  where $T_h(\mathbf{x})$ denotes the translation of $\mathbf{x}$ by $h$, cf. \cite[Section 9.4.6]{FV10}.
\end{lemma}

\begin{proof}
 Recall that
 \begin{align*}
  \rho_{p-\text{var};[S,T]}(\mathbf{x},\mathbf{y}) = \sup_{D \in \mathcal{P}([S,T])} \left( \sum_{t_i \in D} |x_{t_i,t_{i+1}} - y_{t_i,t_{i+1}}|^p \right)^{\frac{1}{p}} + \sup_{D \in \mathcal{P}([S,T])} \left( \sum_{t_i \in D} |\mathbf{x}^2_{t_i,t_{i+1}} - \mathbf{y}^2_{t_i,t_{i+1}}|^{p/2} \right)^{\frac{2}{p}}.
 \end{align*}
 Therefore, we immediately obtain
 \begin{align*}
  \rho_{p-\text{var};[S,T]}(T_h(\mathbf{x}),\mathbf{x}) \leq \| h\|_{q-\text{var};[S,T]} + \sup_{D \in \mathcal{P}([S,T])} \left( \sum_{t_i \in D} |T_h(\mathbf{x})^2_{t_i,t_{i+1}} - \mathbf{x}^2_{t_i,t_{i+1}}|^{p/2} \right)^{\frac{2}{p}}.
 \end{align*}
 Concerning the second term, fix some $D \in \mathcal{P}([S,T])$. We have
  \begin{align*}
   \sum_{t_i \in D} |T_h(\mathbf{x})^2_{t_i,t_{i+1}} - \mathbf{x}_{t_i,t_{i+1}}^2|^{p/2} = \sum_{t_i \in D} \left|\int_{\Delta_{{t_i},{t_{i+1}}}^2} d(x+h) \otimes d(x+h) - \int_{\Delta_{{t_i},{t_{i+1}}}^2} dx \otimes dx \right|^{p/2}
  \end{align*}
  and
  \begin{align*}
   &\left|\int_{\Delta_{{t_i},{t_{i+1}}}^2} d(x+h) \otimes d(x+h) - \int_{\Delta_{{t_i},{t_{i+1}}}^2} dx \otimes dx \right|\\ 
   \leq\ &\left| \int_{\Delta_{t_i,t_{i+1}}^2} dh \otimes d(x+h) \right| + \left| \int_{\Delta_{t_i,t_{i+1}}^2} dx \otimes dh \right| \\
   \leq\ &C_{p,q} \|h\|_{q-\text{var};[t_i,t_{i+1}]} \left( \|x+h\|_{p-\text{var};[t_i,t_{i+1}]} + \|x\|_{p-\text{var};[t_i,t_{i+1}]} \right)
  \end{align*}
 by the estimates for the Young integral. From H\"older's inequality,
 \begin{align*}
   \sum_{t_i \in D} |T_h(\mathbf{x})^2_{t_i,t_{i+1}} - \mathbf{x}_{t_i,t_{i+1}}^2|^{p/2} &\leq C_{p,q} \left( \sum_{t_i} \|h\|_{q-\text{var};[t_i,t_{i+1}]}^q \right)^{\frac{p}{2q}} \left( \sum_{t_i} \|x+h\|_{p-\text{var};[t_i,t_{i+1}]}^p + \|x\|_{p-\text{var};[t_i,t_{i+1}]}^p \right)^{\frac{1}{2}} \\
   &\leq C_{p,q} \|h\|_{q-\text{var};[S,T]}^{p/2}(\|x+h\|_{p-\text{var};[S,T]}^{p/2} + \|x\|_{p-\text{var};[S,T]}^{p/2} )
  \end{align*}
  and the result follows from the triangle inequality for the $p$-variation seminorm and standard estimates.
\end{proof}

\begin{lemma}\label{lemma:dist_ito_lyons_sup}
 Let $\mathbf{x}^1 := \mathbf{x}$ and $\mathbf{x}^2 := T_h(\mathbf{x})$ where $\mathbf{x}$ is a weakly geometric $p$-rough path for some $p \in [1,3)$ and $h$ is a path of finite $q$-variation with $\frac{1}{p} + \frac{1}{q} > 1$. Consider the solutions $y^1$ and $y^2$ to the RDEs as in Proposition \ref{prop:loc_lip_integrability_p-var_top} with $f^1 = f^2$ and $y^1_S = y^2_S$.
 Then
 \begin{align*}
  d_{p-\text{var};[S,T]}(y^1,y^2) \leq C \exp \{ C (N_1(\mathbf{x};[S,T]) + 1) \} (\|h\|_{q-\text{var};[S,T]} \vee \|h\|_{q-\text{var};[S,T]}^q)
 \end{align*}
 where $C$ is a constant depending on $p,q,\theta$ and $\beta$.
\end{lemma}

\begin{proof}
 We will only consider the case $p \in [2,3)$, the case $p \in [1,2)$ is similar (and easier). Let $\|h\|_{q-\text{var};[S,T]} \leq 1$. We claim that
 \begin{align}\label{eqn:claim_1_small_h}
 	 d_{p-\text{var};[S,T]}(y^1,y^2) \leq C \exp\left\{C(N_{1}(\mathbf{x};[S,T]) + 1)\right\} \|h\|_{q-\text{var};[S,T]}
 \end{align} 
 holds for some constant $C$. Indeed: From Proposition \ref{prop:loc_lip_integrability_p-var_top} we know that for every $\alpha > 0$,
 \begin{align*}
    d_{p-\text{var};[S,T]}(y^1,y^2) \leq &C ( \|\mathbf{x} \|_{p-\text{var};[S,T]} + \| T_h(\mathbf{x}) \|_{p-\text{var};[S,T]} + 1 )\\
    &\times \exp\left\{C(N_{\alpha}(\mathbf{x};[S,T]) + N_{\alpha}(T_h(\mathbf{x});[S,T]) + 1)\right\} \rho_{p-\text{var};[S,T]}(\mathbf{x},T_h(\mathbf{x})).
 \end{align*}
  Using \cite[Theorem 9.33]{FV10}, \cite[Lemma 5]{BFRS16}, \cite[Lemma 1]{FR12b} and the assumption $\|h\|_{q-\text{var};[S,T]} \leq 1$ shows that
 \begin{align*}
    d_{p-\text{var};[S,T]}(y^1,y^2) \leq C ( \|\mathbf{x} \|_{p-\text{var};[S,T]} + 1 )\exp\left\{C(N_{1}(\mathbf{x};[S,T]) + 1)\right\} \rho_{p-\text{var};[S,T]}(\mathbf{x},T_h(\mathbf{x}))
 \end{align*}
 for a larger constant $C$ and $\alpha$ chosen appropriately. Applying Lemma \ref{lemma:estimates_diff_young_shift_inhom_rp_metric} shows \eqref{eqn:claim_1_small_h}, using the estimate $\|\mathbf{x}\|_{p-\text{var};[S,T]} \leq N_1(\mathbf{x};[S,T]) + 1$ which was proven in \cite[Lemma 4]{FR12b}.
 
 Now let $\|h\|_{q-\text{var};[S,T]} \geq 1$. In this case,
 \begin{align*}
  d_{p-\text{var};[S,T]}(y^1,y^2) &\leq \|y^1\|_{p-\text{var};[S,T]} + \|y^2\|_{p-\text{var};[S,T]} \\
  &\leq C(N_{\alpha}(\mathbf{x};[S,T]) + N_{\alpha}(T_h(\mathbf{x});[S,T]) + 1)
 \end{align*}
 using the deterministic estimates for the It\=o--Lyons map proven in \cite{FR12b}. With \cite[Lemma 11.12]{FH14}, we conclude that
 \begin{align*}
 	d_{p-\text{var};[S,T]}(y^1,y^2) \leq C(N_{1}(\mathbf{x};[S,T])+1) \|h\|_{q-\text{var};[S,T]}^q
 \end{align*}
 for a larger constant $C$.

\end{proof}

We come back to our original setup. Assume that $\gamma$ is a Gaussian measure on the Borel sets of the Banach space $C_0([0,T],\R^m)$ induced by a continuous $\R^m$-valued Gaussian process $X$. As usual, we denote the corresponding Cameron-Martin space by $\mathcal{H}$. Assume that there is some $p \in [1,3)$ and a measurable lift map $S \colon C_0 \to \mathcal{D}_g^{0,p}$ such that the diagram \eqref{eqn:comm_diag} commutes on a set of full $\gamma$-measure. Will now make further assumptions on our lift map $S$: Suppose that
\begin{itemize}
 \item[(i)] There is a continuous embedding
\begin{align*}
 \iota \colon \mathcal{H} \hookrightarrow C^{q-\text{var}}([0,T],\R^d);\quad 1\leq q \leq p
\end{align*}
with $\frac{1}{p} + \frac{1}{q} > 1$ (note that this implies $1 \leq q < 2$ when $p \geq 2$).
  \item[(ii)] The set
  \begin{align*}
   \left\{ x \in C_0 \ |\ S(x + h) = T_h(S(x)) \text{ for all } h \in \mathcal{H} \right\} 
  \end{align*}
  has full $\gamma$-measure.

\end{itemize}

\begin{remark}
 Assumption (i) and (ii) are trivially satisfied for $p \in [1,2)$. More generally, they hold if the covariance of the corresponding $\R^m$-valued Gaussian process has \emph{mixed $(1,\rho)$-variation} for some $\rho \in [1,3/2)$ with $q = 2(1/\rho + 1)^{-1}$, cf. \cite[p. 688]{FGGR16} for the definition of mixed variation and \cite[Theorem 1.1]{FGGR16} and \cite[Lemma 15.58]{FV10} for the corresponding results. A list of processes which satisfy this condition can also be found in \cite{FGGR16}. In particular, they hold for the Stratonovich lift of the Brownian motion with $q = 1$.
\end{remark}

Under these two conditions, the following Proposition is an immediate consequence of Lemma \ref{lemma:dist_ito_lyons_sup}.

\begin{proposition}\label{cor:loc_lip_mult_noise}
 Consider the RDEs as in Proposition \ref{prop:loc_lip_integrability_p-var_top}.
 
 Then
 \begin{align*}
    d_{p-\text{var}}(y^1,y^2) \leq L(x^1) (\|\iota\|_{\mathcal{H} \hookrightarrow C^{q-\text{var}}} \vee \|\iota\|_{\mathcal{H} \hookrightarrow C^{q-\text{var}}}^q) (d_{\mathcal{H}}(x^1,x^2) \vee d_{\mathcal{H}}(x^1,x^2)^q)
 \end{align*}
  for all $x^1,x^2 \in C_0$ where
  \begin{align*}
   L(x) = C \exp \big\{ C (N_1(S(x);[0,T]) + 1) \big\}
  \end{align*}
  and $C$ is a constant depending on $p,q,\theta$ and $\beta$.
\end{proposition}
The next theorem is our main result for the multiplicative case.

\begin{theorem}\label{thm:main_result_multiplicative_sdes}
 Let $Y$ be the solution to the SDE \eqref{eqn:mult_case_gaussian_sde} driven by the Gaussian process $X$ defined pathwise via the diagram \eqref{eqn:comm_diag} and let $\mu$ be the law of $Y$. Assume that $q = 1$. Then for every $\varepsilon > 0$ there is a constant $C$ depending on $\varepsilon$, $p$, $\theta$ and $\beta$ such that for every $\nu \in P(C_{\xi})$,
 \begin{align*}
   \inf_{\pi \in \Pi(\nu,\mu)} \left(\int_{C_{\xi} \times C_{\xi}} d_{p-\text{var}}(x,y)^{2 - \varepsilon}\, d\pi(x,y)\right)^{\frac{1}{2 - \varepsilon}} \leq C \|\iota\|_{\mathcal{H} \hookrightarrow C^{1-\text{var}}} \sqrt{H(\nu\,|\,\mu)}.
 \end{align*}
\end{theorem}

\begin{proof}
  From \cite[Theorem 6.3]{CLL13} we know that $N_1(S;[0,T])$ has Gaussian tails w.r.t. $\gamma$, hence $\| \exp \{ C (N_1(S;[0,T])+1) \}\|_{L^q(\gamma)} < \infty$ for every $q \in [1,\infty)$.  The assertion follows from Theorem \ref{thm:talagrand_strong_form_linear}, Proposition \ref{cor:loc_lip_mult_noise} and the contraction principle Lemma \ref{lemma:transp_ineq_trans_under_loc_lip_maps}.

\end{proof}

\begin{example}
 Let us come back to the bifractional Brownian motion $B^{H,K} \colon [0,T] \to \R^m$ already considered in Example \ref{ex:bifBm}. In \cite{RT06} and \cite{KRT07}, it was shown that in the case $2HK = 1$, the process has many similarities to the usual Brownian motion. The same holds true here: From \cite[Example 2.12]{FGGR16}, we know that the covariance of the bifractional Brownian motion has mixed $(1,\rho)$-variation for $\rho = (2HK)^{-1}$. In particular, we can choose $q = 1$ in the case $2HK = 1$, and Theorem \ref{thm:main_result_multiplicative_sdes} applies. Therefore, we can \emph{almost} (i.e. modulo an $\varepsilon$-correction) deduce Talagrand's transport inequality in this case. However, for the Brownian motion (which we obtain for the choice $K = 1$ and $H = 1/2$), it is known that Talagrand's inequality holds for the uniform distance (which is smaller than the $p$-variation distance) even without $\varepsilon$-correction, cf. \cite{Ust12}. It remains an open problem how to obtain the full $2$-
transport inequality for diffusions driven by a multidimensional Brownian motion without using the Girsanov transformation.  
\end{example}

\section{Tail estimates for functionals}\label{sec:tail_estimates_for_functionals}

In the following, we aim to motivate why it is useful to have $p$-transportation--cost inequalities for $p > 1$. This section is independent of the former one and may be interesting in its own right.

It is well known that transportation--cost inequalities imply Gaussian measure concentration. This was first disovered by Marton (\cite{Mar86}, \cite{Mar96}). In \cite{DGW04}, it was shown that for $p = 1$, the converse is true: Gaussian tails imply the $1$-transportation--cost inequality. In the case of a Gaussian Banach space $(E,\mathcal{H},\gamma)$, it is a classical result (cf. \cite[4.5.6. Theorem]{Bog98}) that $\mathcal{H}$-Lipschitz functions on Gaussian spaces have Gaussian tails. This result was further generalized in \cite[Theorem 17]{DOR15} (cf. also \cite{FO10} and \cite[Theorem 11.7]{FH14}) where it was shown that the linear growth of a function in $\mathcal{H}$-direction already implies that it has Gaussian tails. More precisely, if there is a constant $\sigma > 0$ and a measurable map $g \colon E \to [0,\infty]$ for which $g < \infty$ on a set of positive $\gamma$-measure such that $f \colon E \to [0,\infty]$ satisfies
\begin{align}\label{eqn:linear_growth_H_direction}
 f(x + h) \leq g(x) + \sigma | h |_{\mathcal{H}}
\end{align}
for all $x$ on a set of full $\gamma$-measure and all $h \in \mathcal{H}$, then $f$ has Gaussian tails. In the following, we will prove an abstract result which will imply that we may even choose $\sigma$ random in \eqref{eqn:linear_growth_H_direction} and still obtain Gaussian tails for $f$.


\begin{theorem}\label{theorem:tail_estimates_general_form}
 Let $E$ be a linear Polish space and let $\mu$ be a probability measure defined on its Borel $\sigma$--algebra. Assume that there is a normed subspace $\mathcal{U} \subseteq E$ and let $d_{\mathcal{U}} \colon E \times E \to [0,\infty]$ be defined as
 \begin{align*}
  d_{\mathcal{U}}(x,y) = \begin{cases}
                          | x - y |_{\mathcal{U}} &\text{ if } x - y \in \mathcal{U} \\
                          + \infty &\text{ otherwise.}
                         \end{cases}
 \end{align*}
Assume that there is a $p \in [1, \infty)$ and a constant $C$ such that for every $\nu \in P(E)$,
\begin{align*}
 \inf_{\pi \in \Pi(\nu,\mu)} \left(\int_{E \times E} d_{\mathcal{U}}(x,y)^{p}\, d\pi(x,y)\right)^{\frac{1}{p}} \leq  \sqrt{ C H(\nu\,|\,\mu)}.
\end{align*}
  Let $(F,d)$ be some metric space and let $f \colon E \to F$ be measurable w.r.t. the Borel $\sigma$--algebra. Choose $r_0 \geq 0$ and some element $e \in E$ such that 
  \begin{align*}
   \mu\left\{ x \in E\, : \, d(f(x),e) \leq r_0 \right\} =: a > 0.
  \end{align*}
  Assume that there are measurable functions $g, \sigma \colon E \to [0,\infty]$ such that
  \begin{align*}
   d(f(x + h),e) \leq g(x) + \sigma(x)|h|_{\mathcal{U}}
  \end{align*}
  holds for every $x \in E$ and every $h \in \mathcal{U}$, and assume that $g \in L^1(\mu)$ and $\sigma \in L^q(\mu)$ where $q \in (1,\infty]$ is chosen such that $\frac{1}{q} + \frac{1}{q} = 1$. 
  
  Then 
  \begin{align*}
   \mu \left\{ x \in E \, :\, d(f(x),e) > r \right\} \leq \exp \left\{ - \frac{(r - r_1)^2}{C \|\sigma\|_{L^q(\mu)}^2} \right\}
  \end{align*}
  for all $r \geq r_1$ where $r_1 = r_0 + 4\|g\|_{L^1(\mu)} +  \|\sigma\|_{L^q(\mu)} \sqrt{2C \log(a^{-1})}$. In particular, the random variable $d(f(\cdot),e) \colon E \to [0,\infty)$ has Gaussian tails.

\end{theorem}

\begin{proof}
 For $x, y \in E$ set
 \begin{align*}
  d_f(x,y) = d(f(x),f(y)).
 \end{align*}
 For any measurable set $A \subseteq E$ and $r \geq 0$ we define
 \begin{align*}
  A^r := \left\{ x \in E\, :\, \text{there is an }\bar{x} \in A\text{ such that } d_f(x,\bar{x}) \leq r \right\}.
 \end{align*}
 Fix some $r\geq 0$ and set $B := (A^r)^c$. Assume first that $A$ and $B$ have positive measure. On $E$, we define the measures
 \begin{align*}
  d\mu_A := \frac{\mathbbm{1}_A}{\mu(A)}d\mu \quad \text{and} \quad d\mu_B := \frac{\mathbbm{1}_B}{\mu(B)}d\mu.
 \end{align*}
 Then
 \begin{align*}
  r &\leq \inf_{\pi \in \Pi(\mu_A,\mu_B)} \int_{E \times E} d_f(x,y) \, d\pi(x,y) \\
  &\leq \inf_{\pi \in \Pi(\mu_A,\mu)} \int_{E \times E} d_f(x,y) \, d\pi(x,y) + \inf_{\pi \in \Pi(\mu_B,\mu)} \int_{E \times E} d_f(x,y) \, d\pi(x,y)
 \end{align*}
 where we used symmetry and the triangle inequality for optimal transportation costs. (The triangle inequality can be deduced, in our case, exactly as for the usual Wasserstein metric using the Gluing Lemma \cite[Lemma 7.6]{Vil03} and the triangle inequality for $d_f$, cf. \cite[Theorem 7.3]{Vil03}.) If $x - y = h \in \mathcal{U}$, we obtain by assumption
 \begin{align*}
  d_f(x,y) \leq 2 g(y) + \sigma(y) |h|_{\mathcal{U}}.
 \end{align*} 
  This implies that for all $x,y \in E$,
  \begin{align*}
   d_f(x,y) \leq 2 g(y) + \sigma(y) d_{\mathcal{U}}(x,y).
  \end{align*}
  It follows that
  \begin{align*}
   r &\leq \inf_{\pi \in \Pi(\mu_A,\mu)} \int_{E \times E} \sigma(y) d_{\mathcal{U}}(x,y) \, d\pi(x,y) + \inf_{\pi \in \Pi(\mu_B,\mu)} \int_{E \times E} \sigma(y) d_{\mathcal{U}}(x,y) \, d\pi(x,y) + 4 \| g \|_{L^1} \\
   &\leq \|\sigma\|_{L^q} \left( \inf_{\pi \in \Pi(\mu_A,\mu)} \left( \int_{E \times E} d_{\mathcal{U}}(x,y)^p \, d\pi(x,y) \right)^{\frac{1}{p}} + \inf_{\pi \in \Pi(\mu_B,\mu)} \left( \int_{E \times E} d_{\mathcal{U}}(x,y)^p \, d\pi(x,y) \right)^{\frac{1}{p}}\right) +  4 \| g \|_{L^1} \\
   &\leq  \|\sigma\|_{L^q} \left( \sqrt{C H(\mu_A\,|\,\mu)} +  \sqrt{C H(\mu_B\,|\,\mu)} \right) + 4 \| g \|_{L^1} \\
   &= \|\sigma\|_{L^q} \left( \sqrt{C \log(\mu(A)^{-1})} + \sqrt{C \log(\mu(B)^{-1})} \right) + 4 \| g \|_{L^1}.
  \end{align*}
  Rearranging terms, we see that
  \begin{align*}
    \mu(A^r) \geq 1 - \exp\left\{ - \frac{(r - \hat{r})^2}{C \| \sigma \|_{L^q}^2} \right\} 
  \end{align*}
  for every $r \geq \hat{r}$ where $\hat{r} :=  \|\sigma\|_{L^q}\sqrt{C \log(\mu(A)^{-1})} + 4 \| g \|_{L^1}$. Now set
  \begin{align*}
   A := \{ x\in E\, :\, d(f(x),e) \leq r_0 \}.
  \end{align*}
  By assumption, $\mu(A) = a > 0$. For every $r \geq 0$, we have
  \begin{align*}
   A^r \subseteq \{ x \in E \, :\, d(f(x),e) \leq r_0 + r\}.
  \end{align*}
  If $\mu(B) = 0$, it follows that $\{ x \in E \, :\, d(f(x),e) \leq r_0 + r\}$ has full measure. In other words, $d(f(\cdot),e)$ is bounded almost surely and the claimed estimate is trivial. If $\mu(B) > 0$, we can use our calculations above to conclude that
  \begin{align*}
   1 - \exp\left\{ - \frac{(r - \hat{r})^2}{C \|\sigma\|_{L^q}^2}  \right\} \leq \mu\{ x \in E\, :\, d(f(x),e) \leq r_0 + r\}
  \end{align*}
  holds for every $r \geq \hat{r}$ and the claim follows.

\end{proof}

In the Gaussian case, Theorem \ref{thm:talagrand_strong_form_linear} immediately implies

\begin{corollary}\label{cor:gen_fernique}
 Let $(F, \mathcal{H},\gamma)$ be a Gaussian Fr\'echet space and $f \colon F \to [0,\infty]$ be measurable. Assume that there are nonnegative random variables $g \in L^1(\gamma)$ and $\sigma \in L^2(\gamma)$ such that
 \begin{align*}
  f(x + h) \leq g(x) + \sigma(x)|h|_{\mathcal{H}}
 \end{align*}
 holds for every $x \in F$ and $h \in \mathcal{H}$.
 
 Then $f$ has Gaussian tails.

\end{corollary}

\begin{remark}
 Corollary \ref{cor:gen_fernique} is more universal than the \emph{Generalized Fernique Theorem} proven in \cite[Theorem 17]{DOR15} and \cite[Theorem 11.7]{FH14} (cf. also \cite{FO10}) since it allows $\sigma$ to be an $L^2(\gamma)$-random variable. Moreover, our proof does not rely on the Borell-Sudakov-Cirelson inequality (\cite{Bor75}, \cite{SC74}) and can be applied in more general frameworks whenenver transport inequalities are available. Non-Gaussian examples include the law of diffusions driven by Gaussian processes, as was shown in this work.
\end{remark}

\section{Appendix}

\subsection{A generalized contraction principle}

The next Lemma is a generalization of \cite[Lemma 2.1]{DGW04}.

\begin{lemma}\label{lemma:transp_ineq_trans_under_loc_lip_maps}
	Let $(X,\mathcal{F})$ be a measurable space on which regular conditional distributions exist and let $c \colon X \times X \to \R_+ \cup\{+ \infty\}$ be a measurable function. Assume that there is a measure $\mu \in P(X)$ such that
	\begin{align*}
		\inf_{\pi \in \Pi(\nu,\mu)} \left( \int_{X \times X} c(x,y)^p\, d\pi(x,y) \right)^{\frac{1}{p}} \leq \sqrt{C H(\nu\,|\,\mu)}
	\end{align*}
	holds for every $\nu \in P(X)$ where $C$ is some constant and $p \in [1,\infty)$. Let $(Y,\mathcal{G})$ be another measurable space, $\tilde{c} \colon Y \times Y \to \R_+ \cup\{+ \infty\}$ be a measurable function and assume that there is a measurable function $\Psi \colon X \to Y$ for which
	\begin{align*}
		\tilde{c}(\Psi(x),\Psi(y)) \leq L(y) c(x,y)
	\end{align*}
	holds for every $x,y \in X_0$ where $X_0 \subseteq X$ has full measure w.r.t. $\mu$ and $L \colon X \to \R \cup\{+ \infty\}$ is another measurable function. Set $\tilde{\mu} := \mu \circ \Psi^{-1}$. Then for every $\tilde{p} \in [1,p]$,
	\begin{align*}
		\inf_{\tilde{\pi} \in \Pi(\tilde{\nu},\tilde{\mu})} \left( \int_{Y \times Y} \tilde{c}(x,y)^{\tilde{p}} \, d\tilde{\pi}(x,y) \right)^{\frac{1}{\tilde{p}}} \leq  \|L\|_{L^q(\mu)} \sqrt{C H(\tilde{\nu}\,|\,\tilde{\mu})}
	\end{align*}
	holds for very $\tilde{\nu} \in P(Y)$ where $q \in (1, \infty]$ is chosen such that $\frac{1}{q} + \frac{1}{p} = \frac{1}{\tilde{p}}$. 
\end{lemma}

\begin{proof}
	W.l.o.g. we may assume $C = 1$. Let $\tilde{\nu} \in P(Y)$ and assume that $H(\tilde{\nu}\,|\,\tilde{\mu}) < \infty$. Choose $\nu \in P(X)$ such that $\tilde{\nu} = \nu \circ \Psi^{-1}$ and  $ \nu \ll \mu$ (note that there is at least one $\nu$ which fulfills this condition; e.g. $\nu_0(dx) := \frac{d\tilde{\nu}}{d\tilde{\mu}}(\Psi(x)) \mu(dx)$). Then
	\begin{align*}
		\inf_{\tilde{\pi} \in \Pi(\tilde{\nu},\tilde{\mu})} \int \tilde{c}(x,y)^{\tilde{p}} \, d\tilde{\pi}(x,y) 		
		&\leq \inf_{\pi \in \Pi(\nu,\mu)} \int_{Y \times Y} \tilde{c}(x,y)^{\tilde{p}} \, d(\pi \circ (\Psi\times \Psi)^{-1})(x,y) \\
		&= \inf_{\pi \in \Pi(\nu,\mu)} \int_{X \times X} \tilde{c}(\Psi(x),\Psi(y))^{\tilde{p}} \, d \pi(x,y).
	\end{align*}
	Since $\nu \ll \mu$, $X_0 \times X_0$ has full measure for every $\pi \in \Pi(\nu,\mu)$, therefore
	\begin{align*}
		\inf_{\pi \in \Pi(\nu,\mu)} \int_{X \times X} \tilde{c}(\Psi(x),\Psi(y))^{\tilde{p}} \, d \pi(x,y) 
		&\leq \inf_{\pi \in \Pi(\nu,\mu)} \int_{X \times X} (L(y) c(x,y))^{\tilde{p}} \, d \pi(x,y) \\
		&\leq \|L\|_{L^q(\mu)}^{\tilde{p}} \inf_{\pi \in \Pi(\nu,\mu)} \left( \int_{X \times X} c(x,y)^p \, d \pi(x,y) \right)^{\frac{\tilde{p}}{p}}
	\end{align*}
	by H\"older's inequality. The assertion follows from the identity
	\begin{align}
  H(\tilde{\nu}\,|\,\tilde{\mu}) = \inf \{ H(\nu\,|\,\mu)\ |\ \nu \in P(X)\text{ s.t. } \nu \circ \Psi^{-1} = \tilde{\nu} \}
 \end{align}
  which holds under the assumption that regular conditional distributions exist on $(X, \mathcal{F})$, see \cite[Lemma 2.1]{DGW04}.
\end{proof}



\bibliographystyle{alpha}
\bibliography{refs}

\end{document}